\documentclass{article}
\usepackage{amsthm,amsfonts,amsbsy,amssymb,amsmath,graphicx}
\usepackage{mathrsfs}
\usepackage[all]{xy}

\DeclareMathAlphabet{\mathdutchcal}{U}{dutchcal}{m}{n}

\newtheorem{theorem}{Theorem}
\newtheorem*{theorem*}{Theorem}
\newtheorem{lemma}{Lemma}
\newtheorem{corollary}{Corollary}

\newtheorem{proposition}{Proposition}

\theoremstyle{definition}
\newtheorem{definition}{Definition}

\theoremstyle{remark}
\newtheorem{remark}{Remark}
\newtheorem{example}{Example}

\def\Z{{\mathbb Z}}

\def\R{{\mathbb R}}

\DeclareMathOperator{\sgn}{\mathrm{sgn}}

\title{Multicrossing complex of knot and secant classes}
\author{Igor Nikonov}
\date{}

\begin{document}

\maketitle

\begin{abstract}
    The multicrossing homology of a knot can be considered as a generalization of the homology of its fundamental quandle. We show that certain sums of knot secants define invariant classes in the multicrossing homology. The 2-secant class generalizes such invariants as linking number and quandle cocycle invariants.
\end{abstract}

\section{Introduction}

One of the simplest knot invariants is the linking number, which was originally defined by C.F. Gauss as a double integral. Today, the linking number is more often described as the sum of the signs of crossings between the components of the link. Many other, more sophisticated knot invariants can be described using a similar approach —- by summing certain values assigned to crossings over all crossings in the diagram. These include cocyclic quandle invariants and index polynomials of virtual knots. The author's work~\cite{Nca} demonstrated that all these invariants come from a certain class of the multicrossing homology of the knot.

The aim of this work is to introduce two other invariant classes of homology for secant knots, defined using $3$- and $4$-secants of links.

Since crossings in standard link diagrams can be described with $2$-secants, all three homology classes can be expressed by a single formula:
\begin{equation}\label{eq:secant_class}
\sigma_n(L)=\sum_{s\in\mathcal S^{adm}_n(L)}\sgn(s)\cdot s,\quad n=2,3,4,
\end{equation}
where $\mathcal S^{adm}_n(L)$ represents the set of admissible $n$-secants for the link $L$, and $\sgn(s)$ denotes the sign of the secant. The admissibility condition depends on $n$: for $n=2$, the set $\mathcal S^{adm}_2(L)$ consists of ``vertical'' $2$-secants parallel to a given direction $\vec\nu$; for $n=3$, ``horizontal'' $3$-secants perpendicular to the direction $\vec\nu$ are considered; and for $n=4$, the set $\mathcal S^{adm}_4(L)$ includes all possible $4$-secants of the link $L$.

Like vertical $2$-secants correspond to the crossings of the link diagram, horizontal $3$-secant and $4$-secants can be considered as ``crossings'' of some (non Reidemeister) diagram. Such diagrams arise as a result of applying a projector to a link in $\R^3$:
\begin{equation*}
\xymatrix{
*+[F]{\txt{knot}} \ar@{-}[r] & *+[F-:<12pt>]{\txt{projector}} \ar[r] & *+[F]{\txt{knot diagram}}
}
\end{equation*}
Besides the Reidemeister projector, in the paper we consider two other examples which use horizontal lines or all lines in $\R^3$ (Fig.~\ref{fig:projectors}). For the definitions see Section~\ref{subsect:projectors}.

\begin{figure}[h]
    \centering
    \includegraphics[width=0.8\linewidth]{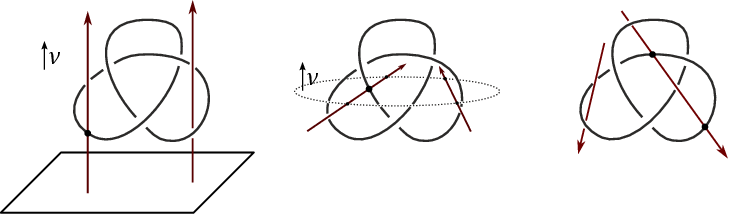}
    \caption{Reidemeister, spinner and omni projectors}
    \label{fig:projectors}
\end{figure}

Diagrams formed by horizontal secants were investagated by Fiedler and Kulin in~\cite{FK1,FK2,FK3}. The study of horizontal $3$-secants is closely connected to the theory of $G^k_n$ groups introduced by V.O. Manturov~\cite{KLM,Man15,FKMN,MN}.

Constructions and invariants related to $4$-secants are discussed in the review~\cite{D}. Budney, Conant, Shannel and Sinha~\cite{BCSS} show that counting of alternating quadrisecants yields the Vassiliev knot invariant of order $2$.

The paper is organized as follows. In Section~\ref{sect:secants} we defines signs of secants and give a geometric interpretation of the signs. Then we discuss three projectors of links in $\R^3$ and singular secants which can appear in generic isotopies of links. Section~\ref{sect:multicrossing_complex} introduces the multicrossing complex and defines the secant classes. In Section~\ref{subsect:component_cocycles} we consider invariants defined by component cocycles of secants and reprove results of Viro~\cite{V} that relate counting of quadrisecants with linking numbers of components of the link. Section~\ref{sect:proof} is devoted to the proof of the main theorem. By analogy with quandle colorings, in Section~\ref{sect:ml_invariants} we $(m,l)$-invariants of links and show they are stronger than the coloring polynomial defined by Eisermann~\cite{E}.

\section{Secants}\label{sect:secants}

\begin{definition}
    For $n\ge 0$, an \emph{$n$-secant} of a link $L$ is an oriented line that intersects the link in $n$ points. A secant is called \emph{transversal} if all the intersections are transversal.

    For a secant $s$, denote the secant with the reversed orientation by $-s$.

    We will say that $s$ is an $n$-secant on arcs $\alpha_1(t)\dots,\alpha_n(t)$ if there are non overlapping local parametrizations $\alpha_1(t)\dots,\alpha_n(t)$ of the link $L$ such that $s\cap L=\{\alpha_1(t_1),\dots,\alpha_n(t_n)\}$ for some  $t_1,\dots,t_n$, and the intersection points are enumerated in the order induced by the orientation of $s$. By default, we will assume that $t_1=\cdots =t_n=0$.

    Note that the reversed secant $-s$ is a secant on arcs $\alpha_n(t),\dots,\alpha_1(t)$.

    Denote the set of transversal $n$-secants of the link by $\mathcal S_n(L)$.
\begin{figure}[h]
    \centering
    \includegraphics[width=0.25\textwidth]{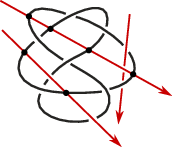}
    \caption{A $0$-secant, a $2$-secant and a $4$-secant }
    \label{fig:secants}
\end{figure}
\end{definition}

\subsection{Signs of secants}\label{subsect:secant_sign}

\subsubsection{Sign of $2$-secants}
For a pair of spatial curves $\alpha_1(t)$, $\alpha_2(t)$ denote
\begin{equation}\label{eq:delta2}
 \Delta_2(\alpha_1,\alpha_2)=(\dot\alpha_1(0),\dot\alpha_2(0),\alpha_1(0)-\alpha_2(0)).
\end{equation}

\begin{definition}\label{def:2secant_sign}
Let $s$ be a bisecant on arcs $\alpha_1(t)$ and $\alpha_2(t)$. The \emph{sign} of the bisecant $s$ is $\sgn(s)=\sgn\Delta_2(\alpha_1,\alpha_2)$.
\end{definition}

Below we will also use the notation $\sgn(\alpha_1,\alpha_2)$ for the sign of the bisecant on $\alpha_1(t)$ and $\alpha_2(t)$.

From the definition we have the following statements.
\begin{proposition}\label{prop:2secant_delta}
\begin{enumerate}
    \item $\Delta_2(\alpha_2,\alpha_1)=\Delta_2(\alpha_1,\alpha_2)$;
    \item $\Delta_2(-\alpha_1,\alpha_2)=-\Delta_2(\alpha_1,\alpha_2)$ where $(-\alpha_1)(t)=\alpha_1(-t)$.
\end{enumerate}
\end{proposition}
\begin{corollary}\label{cor:2secant_inverse_sign}
    For a $2$-secant $s\in\mathcal S_2(L)$, $\sgn(-s)=\sgn(s)$.
\end{corollary}

\subsubsection{Sign of horizontal $3$-secants}

Fix a unit vector $\nu$ in $\R^3$.

Let $\alpha_1(t)$, $\alpha_2(t)$, $\alpha_3(t)$ be spatial curves such that the points $\alpha_1(0), \alpha_2(0), \alpha_3(0)$ are collinear. Denote
\begin{multline}\label{eq:delta3}
\Delta_3(\alpha_1,\alpha_2,\alpha_3)= (\dot\alpha_1(0),\nu)(\dot\alpha_2(0),\nu)(\dot\alpha_3(0),\alpha_1(0)-\alpha_2(0),\nu)+\\
    (\dot\alpha_1(0),\nu)(\dot\alpha_3(0),\nu)(\dot\alpha_2(0),\alpha_3(0)-\alpha_1(0),\nu)+\\
    (\dot\alpha_2(0),\nu)(\dot\alpha_3(0),\nu)(\dot\alpha_1(0),\alpha_2(0)-\alpha_3(0),\nu).
\end{multline}

\begin{definition}\label{def:3secant_sign}
Let $s$ be a trisecant on arcs $\alpha_1(t)$, $\alpha_2(t)$ and $\alpha_3(t)$. The \emph{sign} of the trisecant $s$ is $\sgn(s)=\sgn\Delta_3(\alpha_1,\alpha_2,\alpha_3)$.
\end{definition}

The statements below immediately follow from the definition.
\begin{proposition}\label{prop:3secant_delta}
\begin{enumerate}
    \item For any permutation $\tau\in\Sigma_3$,
\[
\Delta_3(\alpha_{\tau(1)},\alpha_{\tau(2)},\alpha_{\tau(3)})=(-1)^\tau\Delta_3(\alpha_1,\alpha_2,\alpha_3);
\]
    \item $\Delta_3(-\alpha_1,\alpha_2,\alpha_3)=-\Delta_3(\alpha_1,\alpha_2,\alpha_3)$.

\end{enumerate}
\end{proposition}

\begin{corollary}\label{cor:3secant_inverse_sign}
    For a $3$-secant $s\in\mathcal S_3(L)$, $\sgn(-s)=-\sgn(s)$.
\end{corollary}

\begin{proposition}[Geometric meaning of the trisecant sign]\label{prop:3secant_sign_geometry}
Assume that $\nu$ is the unit vector of the axis $Oz$ and $s$ is a horizontal trisecant on arcs $\alpha_1(t)$, $\alpha_2(t)$ and $\alpha_3(t)$ in a plane $z=z_0$. Assume also that $(\dot\alpha_i,\nu)>0$, $i=1,2,3$.
\begin{enumerate}
    \item Let $\sgn(s)\ne 0$.  Then for a small $\epsilon>0$, the sign $\sgn(s)$ coincides with the orientation of the triangle $A_1A_2A_3$ where $A_i=\alpha_i\cap\{z=z_0+\epsilon\}$ (see Fig.~\ref{fig:positive_trisecant}).
\begin{figure}
    \centering
    \includegraphics[width=0.4\linewidth]{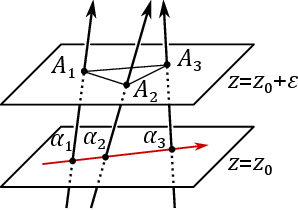}
    \caption{A positive trisecant}
    \label{fig:positive_trisecant}
\end{figure}

\item Denote the surface formed by horizontal $2$-secants on $\alpha_1(t)$ and $\alpha_3(t)$ by $B(\alpha_1,\alpha_3)$. Then $\sgn(s)=0$ if and only if $\alpha_2(t)$ is tangent to $B(\alpha_1,\alpha_3)$ at $t=0$.
\end{enumerate}

\end{proposition}
\begin{proof}
    1. Since $\epsilon$ is small, we linearize the arcs $\alpha_i$. Then $A_i=\alpha_i(0)+\frac{\epsilon}{(\dot\alpha_i(0),\nu)}\dot\alpha_i(0)$. Below we omit the argument $0$. The orientation of the triangle $A_1A_2A_3$ is the sign of the scalar triple product
\begin{multline*}
    (A_2-A_1,A_3-A_1,\nu)=(\alpha_2-\alpha_1+\frac{\epsilon\dot\alpha_2}{(\dot\alpha_2,\nu)}-\frac{\epsilon\dot\alpha_1}{(\dot\alpha_1,\nu)},\alpha_3-\alpha_1+\frac{\epsilon\dot\alpha_3}{(\dot\alpha_3,\nu)}-\frac{\epsilon\dot\alpha_1}{(\dot\alpha_1,\nu)},\nu).
\end{multline*}
We neglect the terms with the multiplier $\epsilon^2$ and use the fact $\alpha_2-\alpha_1\parallel\alpha_3-\alpha_2$, hence, $(\alpha_2-\alpha_1,\alpha_3-\alpha_2,\nu)=0$. Thus, we look for the sign of the expression

\begin{multline*}
    (\alpha_2-\alpha_1,\frac{\epsilon\dot\alpha_3}{(\dot\alpha_3,\nu)}-\frac{\epsilon\dot\alpha_1}{(\dot\alpha_1,\nu)},\nu)+
     (\frac{\epsilon\dot\alpha_2}{(\dot\alpha_2,\nu)}-\frac{\epsilon\dot\alpha_1}{(\dot\alpha_1,\nu)},\alpha_3-\alpha_1,\nu)=\\
     \frac{\epsilon}{(\dot\alpha_1,\nu)(\dot\alpha_2,\nu)(\dot\alpha_3,\nu)}\Delta_3(\alpha_1,\alpha_2,\alpha_3).
\end{multline*}
Then the orientation of $A_1A_2A_3$ is $\sgn\Delta_3(\alpha_1,\alpha_2,\alpha_3)=\sgn(s)$.

2. Assume that $(\dot\alpha_i,\nu)=1$, $i=1,2,3$. Then we can suppose that the parameter $t$ coincides with the coordinate $z$. The surface $B(\alpha_1,\alpha_3)$ has a parametrization
\[
r(z,w)=(1-w)\alpha_1(z)+w\alpha_3(z).
\]
The curve $\alpha_2(t)$ intersects the surface in $r(w_0,z_0)$ where $w_0=\frac{|\alpha_2-\alpha_1|}{|\alpha_3-\alpha_1|}$. The tangent space to $B(\alpha_1,\alpha_3)$ is spanned by the vectors
\[
\frac{\partial r}{\partial z}=\frac{|\alpha_3-\alpha_2|\dot\alpha_1+|\alpha_2-\alpha_1|\dot\alpha_3}{|\alpha_3-\alpha_1|}, \quad
\frac{\partial r}{\partial w}=\alpha_3-\alpha_1.
\]
Then the tangency condition is $(\frac{\partial r}{\partial z},\frac{\partial r}{\partial w},\dot\alpha_2)=0$.

Let $\kappa=\frac{\alpha_3-\alpha_1}{|\alpha_3-\alpha_1|}$ be the direction vector of the trisecant, and $\lambda=\nu\times\kappa$. Then $\kappa,\lambda,\nu$ is a positive orthogonal basis.

Since  the points $\alpha_i$ are collinear, $\alpha_i-\alpha_j=|\alpha_i-\alpha_j|\kappa$, $i>j$. Then
\begin{equation*}
   (\frac{\partial r}{\partial z},\frac{\partial r}{\partial w},\dot\alpha_2)=
   (\dot\alpha_1,\alpha_3-\alpha_2,\dot\alpha_2)+(\dot\alpha_3,\alpha_2-\alpha_1,\dot\alpha_2).
\end{equation*}
Let us show that $(\frac{\partial r}{\partial z},\frac{\partial r}{\partial w},\dot\alpha_2)=\Delta_3(\alpha_1,\alpha_2,\alpha_3)$. By assumption,
\[
\dot\alpha_2=a_1\kappa+a_2\lambda+\nu,\quad a_1,a_2\in\R.
\]
We have
\[
(\dot\alpha_1,\alpha_3-\alpha_2,a_2\lambda)=a_2|\alpha_3-\alpha_2|(\dot\alpha_1,\kappa,\lambda)=a_2|\alpha_3-\alpha_2|(\dot\alpha_1,\nu)=a_2|\alpha_3-\alpha_2|
\]
and $(\dot\alpha_3,\alpha_2-\alpha_1,a_2\lambda)=a_2|\alpha_2-\alpha_1|$. On the other hand,
\[
(\dot\alpha_2,\alpha_1-\alpha_3,\nu)=-a_2|\alpha_3-\alpha_1|(\lambda,\kappa,\nu)=a_2|\alpha_3-\alpha_1|.
\]
Then
\begin{multline*}
    (\frac{\partial r}{\partial z},\frac{\partial r}{\partial w},\dot\alpha_2)=(\dot\alpha_1,\alpha_3-\alpha_2,c_2\lambda)+(\dot\alpha_1,\alpha_3-\alpha_2,\nu)+(\dot\alpha_3,\alpha_2-\alpha_1,c_2\lambda)+\\(\dot\alpha_3,\alpha_2-\alpha_1,\nu)=
    (\dot\alpha_1,\alpha_3-\alpha_2,\nu)+(\dot\alpha_3,\alpha_2-\alpha_1,\nu)+
    (\dot\alpha_2,\alpha_1-\alpha_3,\nu)=\\ \Delta_3(\alpha_1,\alpha_2,\alpha_3).
\end{multline*}
The general case follows from the multi-linearity of $\Delta_3$ in $\dot\alpha_i$.
\end{proof}

\begin{remark}
    From the proof of Proposition~\ref{prop:3secant_sign_geometry} we get another formula for $\Delta_3$:
\begin{equation}
    \Delta_3(\alpha_1,\alpha_2,\alpha_3)=(\dot\alpha_3,\nu)(\dot\alpha_1,\alpha_3-\alpha_2,\dot\alpha_2)+(\dot\alpha_1,\nu)(\dot\alpha_3,\alpha_2-\alpha_1,\dot\alpha_2).
\end{equation}
\end{remark}

\subsubsection{Signs of $4$-secants}

Let $\alpha_1(t)$, $\alpha_2(t)$, $\alpha_3(t)$, $\alpha_4(t)$ be spatial curves such that the points $\alpha_1(0)$, $\alpha_2(0)$, $\alpha_3(0)$, $\alpha_4(0)$ are collinear. Denote
\begin{multline}\label{eq:delta4}
\Delta_4(\alpha_1,\alpha_2,\alpha_3,\alpha_4)=\\
\left|\begin{array}{cc}
       (\dot\alpha_1(0),\dot\alpha_3(0),\alpha_2(0)-\alpha_3(0))  & (\dot\alpha_1(0),\dot\alpha_4(0),\alpha_2(0)-\alpha_4(0)) \\
        (\dot\alpha_2(0),\dot\alpha_3(0),\alpha_1(0)-\alpha_3(0)) & (\dot\alpha_2(0),\dot\alpha_4(0),\alpha_1(0)-\alpha_4(0))
    \end{array}\right|.
\end{multline}

\begin{definition}\label{def:4secant_sign}
Let $s$ be a quadrisecant on arcs $\alpha_1(t)$, $\alpha_2(t)$, $\alpha_3(t)$, and $\alpha_4(t)$. The \emph{sign} of the quadrisecant $s$ is $\sgn(s)=\sgn\Delta_4(\alpha_1,\alpha_2,\alpha_3,\alpha_4)$.
\end{definition}

\begin{proposition}\label{prop:4secant_delta}
\begin{enumerate}
    \item For any permutation $\tau\in\Sigma_4$,
\[
\Delta_4(\alpha_{\tau(1)},\alpha_{\tau(2)},\alpha_{\tau(3)},\alpha_{\tau(4)})=(-1)^\tau\Delta_4(\alpha_1,\alpha_2,\alpha_3,\alpha_4);
\]
    \item $\Delta_4(-\alpha_1,\alpha_2,\alpha_3,\alpha_4)=-\Delta_4(\alpha_1,\alpha_2,\alpha_3,\alpha_4)$.
\end{enumerate}
\end{proposition}
\begin{proof}
    Let $\kappa$ be a direction vector of the line containing the points $\alpha_i(0)$. Denote
\begin{gather*}
    p_1=(\kappa,\dot\alpha_1,\dot\alpha_2)(\kappa,\dot\alpha_3,\dot\alpha_4),\quad q_1=(\alpha_1-\alpha_2,\alpha_3-\alpha_4),\\
    p_2=(\kappa,\dot\alpha_1,\dot\alpha_3)(\kappa,\dot\alpha_2,\dot\alpha_4),\quad q_2=(\alpha_1-\alpha_3,\alpha_2-\alpha_4),\\
    p_3=(\kappa,\dot\alpha_1,\dot\alpha_4)(\kappa,\dot\alpha_2,\dot\alpha_3),\quad q_3=(\alpha_1-\alpha_4,\alpha_2-\alpha_3).
\end{gather*}
Then $q_1-q_2+q_3=0$. Let us show that $p_1-p_2+p_3=0$.
%
First, assume that $\kappa,\dot\alpha_1, \dot\alpha_2$ are independent. Then
\[
\dot\alpha_3=a_1\dot\alpha_1+a_2\dot\alpha_2+a_3\kappa,\quad
\dot\alpha_4=b_1\dot\alpha_1+b_2\dot\alpha_2+b_3\kappa
\]
for some $a_1,a_2,a_3,b_1,b_2,b_3\in\R^3$. Hence,
\[
p_1=(a_1b_2-a_2b_1)(\kappa,\dot\alpha_1,\dot\alpha_2)^2,\
p_2=-a_2b_1(\kappa,\dot\alpha_1,\dot\alpha_2)^2,\
p_3=-a_1b_2(\kappa,\dot\alpha_1,\dot\alpha_2)^2,
\]
and $p_1-p_2+p_3=0$. The general case follows from continuity.

Since points $\alpha_i$ lie on one affine line, $\alpha_i-\alpha_j=(z_i-z_j)\kappa$ for some $z_i\in\R$, $i=1,2,3,4$. Then
\begin{multline*}
\Delta_4(\alpha_1,\alpha_2,\alpha_3,\alpha_4)=(\dot\alpha_1,\dot\alpha_3,\alpha_2-\alpha_3)(\dot\alpha_2,\dot\alpha_4,\alpha_1-\alpha_4)-\\(\dot\alpha_1,\dot\alpha_4,\alpha_2-\alpha_4)(\dot\alpha_2,\dot\alpha_3,\alpha_1-\alpha_3)= (z_1-z_4)(z_2-z_3)(\dot\alpha_1,\dot\alpha_3,\kappa)(\dot\alpha_2,\dot\alpha_4,\kappa)-\\ (z_1-z_3)(z_2-z_4)(\dot\alpha_1,\dot\alpha_4,\kappa)(\dot\alpha_2,\dot\alpha_3,\kappa)=p_2q_3-p_3q_2.
\end{multline*}

Let $\alpha'_1=\alpha_2$, $\alpha'_2=\alpha_1$, $\alpha'_3=\alpha_3$, $\alpha'_4=\alpha_4$. Then
\begin{gather*}
    p'_1=-p_1,\quad p'_2=p_3,\quad p'_3=p_2,\\
    q'_1=-q_1,\quad q'_2=q_3,\quad q'_3=q_2,
\end{gather*}
and
\begin{multline*}
\Delta_4(\alpha_2,\alpha_1,\alpha_3,\alpha_4)=\Delta_4(\alpha'_1,\alpha'_2,\alpha'_3,\alpha'_4)=p'_2q'_3-p'_3q'_2=-(p_2q_3-p_3q_2)=\\ -\Delta_4(\alpha_1,\alpha_2,\alpha_3,\alpha_4).
\end{multline*}

Let $\alpha'_1=\alpha_4$, $\alpha'_2=\alpha_1$, $\alpha'_3=\alpha_2$, $\alpha'_4=\alpha_3$. Then
\begin{gather*}
    p'_1=-p_3,\quad p'_2=-p_2,\quad p'_3=-p_1=p_3-p_2,\\
    q'_1=-q_3,\quad q'_2=-q_2,\quad q'_3=-q_1=q_3-q_2,
\end{gather*}
and
\begin{multline*}
\Delta_4(\alpha_4,\alpha_1,\alpha_2,\alpha_3)=\Delta_4(\alpha'_1,\alpha'_2,\alpha'_3,\alpha'_4)=p'_2q'_3-p'_3q'_2=\\  
-(p_2q_3-p_3q_2)= -\Delta_4(\alpha_1,\alpha_2,\alpha_3,\alpha_4).
\end{multline*}
Since the cycles $(12)$ and $(1234)$ generate the permutation group, we get the first statement.

The second statement is clear.
\end{proof}

\begin{corollary}\label{cor:4secant_inverse_sign}
    For a $4$-secant $s\in\mathcal S_4(L)$, $\sgn(-s)=\sgn(s)$.
\end{corollary}
\begin{proof}
    Let $s$ be a quadrisecants on arcs $\alpha_1(t)$, $\alpha_2(t)$, $\alpha_3(t)$, $\alpha_4(t)$. Then $-s$ is the quadrisecant on the arcs $\alpha_4(t)$, $\alpha_3(t)$, $\alpha_2(t)$, $\alpha_1(t)$, and
\[
\sgn(-s)=\sgn\Delta_4(\alpha_4,\alpha_3,\alpha_2,\alpha_1)=\sgn\Delta_4(\alpha_1,\alpha_2,\alpha_2,\alpha_4)=\sgn(s).
\]
\end{proof}

Let us give a geometric interpretation to the sign of a quadrisecant.

Let $\alpha_1(t)$, $\alpha_2(t)$, $\alpha_3(t)$, $\alpha_4(t)$ be spatial curves such that the points $\alpha_1(0)$, $\alpha_2(0)$, $\alpha_3(0)$, $\alpha_4(0)$ lie on one line $l$ (in any order). Let $\kappa$ be a direction vector of the quadrisecant $l$. Choose a point $O$ on the quadrisecant. Then $\overrightarrow{O\alpha_i(0)}=z_i\kappa$, $z_i\in\R$. Choose $O$ so that $z_i<0$, $i=1,2,3,4$.

Assume that $\kappa,\dot\alpha_1,\dot\alpha_2$ are independent. Choose a plane $\Pi$ that intersect transversely $l$ at the point $O$. Let $\alpha_{12\underline{3}}$ be the parametrized curve formed by intersection of the trisecants on $\alpha_1,\alpha_2,\alpha_3$ with the plane $\Pi$. The underlined index $3$ means we take the parameter on $\alpha_3$ for the parameter on $\alpha_{12\underline{3}}$. Analogously, define $\alpha_{12\underline{4}}$ (Fig.~\ref{fig:4secant_geom}).

\begin{figure}[h]
    \centering
    \includegraphics[width=0.5\linewidth]{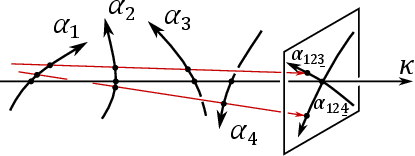}
    \caption{Curves $\alpha_{12\underline{3}}$ and $\alpha_{12\underline{4}}$}
    \label{fig:4secant_geom}
\end{figure}

Let $or(\alpha_{12\underline{3}},\alpha_{12\underline{4}})\in\{-1,0,1\}$ denote the orientation of the triple that consists of the tangent vectors $\dot\alpha_{12\underline{3}}$, $\dot\alpha_{12\underline{4}}$ and $\kappa$. In other words,
\[
or(\alpha_{12\underline{3}},\alpha_{12\underline{4}}) =\sgn(\dot\alpha_{12\underline{3}},\dot\alpha_{12\underline{4}},\kappa).
\]
Analogously, one defines $or(\alpha_1,\alpha_2)$.

\begin{proposition}[Geometric meaning of the quadrisecant sign]\label{prop:4secant_sign_geometry}
\begin{multline*}
\sgn\Delta_4(\alpha_1,\alpha_2,\alpha_3,\alpha_4)=\\
or(\alpha_{12\underline{3}},\alpha_{12\underline{4}})\cdot or(\alpha_1,\alpha_2)\cdot\sgn[(z_1-z_3)(z_2-z_3)(z_1-z_4)(z_2-z_4)].
\end{multline*}
\end{proposition}

\begin{proof}
    First, we consider the case where the plane $\Pi$ is parallel to $\dot\alpha_1$ and $\dot\alpha_2$. Since the left term and the right term of the equation are infinitesimal, we can linearize the curves: $\alpha_1(u)=\alpha_1+\dot\alpha_1u$, $\alpha_2(v)=\alpha_2+\dot\alpha_2v$.

    Consider the coordinate system with the center $O$ and the basis $\dot\alpha_1$, $\dot\alpha_2$ and $\kappa$. The bundle of secants on the curves $\alpha_1(u)$ and $\alpha_2(v)$ projects a point $(x,y,z)$ to a point $(x_0,y_0,0)\in\Pi$ (Fig.~\ref{fig:secant_sheaf}).

\begin{figure}
    \centering
    \includegraphics[width=0.5\linewidth]{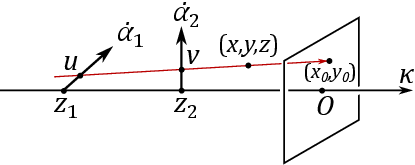}
    \caption{2-secant projection}
    \label{fig:secant_sheaf}
\end{figure}

The secant intersects $\alpha_1$ at a point $(u,0,z_1)$ and intersects $\alpha_2$ at a point $(0,v,z_2)$. Then
\begin{equation}\label{eq:xy_to_uv}
u=\frac{z_1-z_2}{z-z_2}x,\quad v=\frac{z_2-z_1}{z-z_1}y
\end{equation}

and
\[
x_0=-\frac{z_2}{z_1-z_2}u=-\frac{z_2}{z-z_2}x,\quad y_0=-\frac{z_2}{z_2-z_1}v=-\frac{z_1}{z-z_1}y.
\]

Let $\alpha(t)$ be a curve such that $\alpha(0)$ lies on the quadrisecant $l$, i.e. $\alpha(0)=(0,0,z(0))$. Then $\alpha(t)=x(t)\dot\alpha_1+y(t)\dot\alpha_2+z(t)\kappa$ where

\begin{equation}\label{eq:alpha_to_xyz}
x(t)=\frac{(\kappa,\alpha(t),\dot\alpha_2)}{(\dot\alpha_1,\dot\alpha_2,\kappa)},\quad
y(t)=-\frac{(\kappa,\alpha(t),\dot\alpha_1)}{(\dot\alpha_1,\dot\alpha_2,\kappa)},\quad
z(t)=\frac{(\dot\alpha_1,\dot\alpha_2,\alpha(t))}{(\dot\alpha_1,\dot\alpha_2,\kappa)}.
\end{equation}

Then the derivations at $t=0$ are equal
\begin{equation}\label{eq:x0y0_derivation}
\begin{gathered}
    \dot x_0=-\frac{z_2}{z-z_2}\dot x=-\frac{z_2}{z-z_2}\cdot\frac{(\kappa,\dot\alpha,\dot\alpha_2)}{(\dot\alpha_1,\dot\alpha_2,\kappa)},\\
    \dot y_0=-\frac{z_1}{z-z_1}\dot y=\frac{z_1}{z-z_1}\cdot\frac{(\kappa,\dot\alpha,\dot\alpha_1)}{(\dot\alpha_1,\dot\alpha_2,\kappa)}.
\end{gathered}
\end{equation}

Applying~\eqref{eq:x0y0_derivation} to the curves $\alpha_3(t)$ and $\alpha_4(t)$, we see that $or(\alpha_{12\underline{3}},\alpha_{12\underline{4}})$ is the sign of the expression
\[
or(\alpha_1,\alpha_2)\cdot\begin{vmatrix}
                                 x_{03} & x_{04}\\ y_{03} & y_{04}
                             \end{vmatrix}
\]
where
\begin{multline*}
\begin{vmatrix}
x_{03} & x_{04}\\ y_{03} & y_{04}
\end{vmatrix}=
\begin{vmatrix}
   -\frac{z_2}{z_3-z_2}\cdot\frac{(\kappa,\dot\alpha_3,\dot\alpha_2)}{(\dot\alpha_1,\dot\alpha_2,\kappa)} & -\frac{z_2}{z_4-z_2}\cdot\frac{(\kappa,\dot\alpha_4,\dot\alpha_2)}{(\dot\alpha_1,\dot\alpha_2,\kappa)} \\
  \frac{z_1}{z_3-z_1}\cdot\frac{(\kappa,\dot\alpha_3,\dot\alpha_1)}{(\dot\alpha_1,\dot\alpha_2,\kappa)} & \frac{z_1}{z_4-z_1}\cdot\frac{(\kappa,\dot\alpha_4,\dot\alpha_1)}{(\dot\alpha_1,\dot\alpha_2,\kappa)}
\end{vmatrix}=\\
-\frac{z_1z_2}{(\dot\alpha_1,\dot\alpha_2,\kappa)^2(z_1-z_3)(z_2-z_3)(z_1-z_4)(z_2-z_4)}\cdot\\
\cdot\begin{vmatrix}
   (z_1-z_3)(\kappa,\dot\alpha_2,\dot\alpha_3) & (z_1-z_4)(\kappa,\dot\alpha_2,\dot\alpha_4) \\
  (z_2-z_3)(\kappa,\dot\alpha_1,\dot\alpha_3) & (z_2-z_4)(\kappa,\dot\alpha_1,\dot\alpha_4)
\end{vmatrix}=\\
\frac{z_1z_2}{(\dot\alpha_1,\dot\alpha_2,\kappa)^2(z_1-z_3)(z_2-z_3)(z_1-z_4)(z_2-z_4)}\Delta_4(\alpha_1,\alpha_2,\alpha_3,\alpha_4).
\end{multline*}
Since $z_1<0$,  $z_2<0$ and $(\dot\alpha_1,\dot\alpha_2,\kappa)^2>0$,
\begin{multline*}
or(\alpha_{12\underline{3}},\alpha_{12\underline{4}})=\\
or(\alpha_1,\alpha_2)\cdot\sgn\Delta_4(\alpha_1,\alpha_2,\alpha_3,\alpha_4)
\cdot\sgn[(z_1-z_3)(z_2-z_3)(z_1-z_4)(z_2-z_4)].
\end{multline*}
that implies the statement for $\Pi$ parallel to $\dot\alpha_1$ and $\dot\alpha_2$.

The trisecants on $\alpha_1(t)$, $\alpha_2(t)$, $\alpha_3(t)$, and those on $\alpha_1(t)$, $\alpha_2(t)$, $\alpha_4(t)$ form two surfaces which intersect in the line $l$. Then $or(\alpha_{12\underline{3}},\alpha_{12\underline{4}})$ is the sign of the dihedral angle between the surfaces and does not depend on the choice of the transversal plane $\Pi$.
\end{proof}

\begin{remark}\label{rem:4secant_geom}
1. By Definition~\ref{def:2secant_sign}, the sign $\sgn(\alpha_1,\alpha_2)$ of the line $l$ considered as a $2$-secant on $\alpha_1(t)$ and $\alpha_2(t)$ is equal to $or(\alpha_1,\alpha_2)\cdot\sgn(z_1-z_2)$. Hence,
\begin{multline*}
\sgn\Delta_4(\alpha_1,\alpha_2,\alpha_3,\alpha_4)=\\
or(\alpha_{12\underline{3}},\alpha_{12\underline{4}})\cdot \sgn(\alpha_1,\alpha_2)\cdot\sgn[(z_1-z_2)(z_1-z_3)(z_2-z_3)(z_1-z_4)(z_2-z_4)].
\end{multline*}

If $z_1,z_2\ll 0$ then by~\eqref{eq:x0y0_derivation} $\dot x_0\approx\dot x$ and $\dot y_0\approx\dot y$. Hence, we can identify $\alpha_{12\underline{3}}$ with $\alpha_3(t)$ and  $\alpha_{12\underline{4}}$ with $\alpha_4(t)$. Then
\[
or(\alpha_{12\underline{3}},\alpha_{12\underline{4}})=or(\alpha_3,\alpha_4)=\sgn(\alpha_3,\alpha_4)\sgn(z_3-z_4).
\]
and finally
\begin{multline}\label{far_secant_sheaf}
 \sgn\Delta_4(\alpha_1,\alpha_2,\alpha_3,\alpha_4)=\sgn(\alpha_1,\alpha_2)\cdot\sgn(\alpha_3,\alpha_4)\cdot\\
 \cdot\sgn[(z_1-z_2)(z_1-z_3)(z_1-z_4)(z_2-z_3)(z_2-z_4)(z_3-z_4)].
\end{multline}

Note that in general, $or(\alpha_{12\underline{3}},\alpha_{12\underline{4}})$ and $or(\alpha_3,\alpha_4)$ may have different signs, for example, when $\dot\alpha_3=\dot\alpha_4$ and $or(\alpha_3,\alpha_4)=0$ but $or(\alpha_{12\underline{3}},\alpha_{12\underline{4}})\ne 0$.

2. The orientation of the trisecant curve $\alpha_{\underline{1}23}$ induced by the parameter $u$ on $\alpha_1$, differs from the orientation of $\alpha_{12\underline{3}}$ by the sign of $\dot u$. From~\eqref{eq:xy_to_uv} and~\eqref{eq:alpha_to_xyz}
\[
\dot u=\frac{z_1-z_2}{z_3-z_2}\cdot\frac{(\kappa,\dot\alpha_3,\dot\alpha_2)}{(\dot\alpha_1,\dot\alpha_2,\kappa)}=
\left(\frac{z_1-z_2}{z_3-z_2}\right)^2\cdot\frac{(\dot\alpha_3,\dot\alpha_2,\alpha_3-\alpha_2)}{(\dot\alpha_1,\dot\alpha_2,\alpha_1-\alpha_2)},
\]
hence, $\sgn\dot u=\sgn(\alpha_1,\alpha_2)\cdot\sgn(\alpha_2,\alpha_3)$. Analogously, the orientations $\alpha_{1\underline{2}3}$ and  $\alpha_{12\underline{3}}$ differ by
$\sgn\dot v=\sgn(\alpha_1,\alpha_2)\cdot\sgn(\alpha_1,\alpha_3)$.
\end{remark}

\subsection{Knot diagrams and projectors}\label{subsect:projectors}

A conventional way to study knots is to use their diagrams. Diagrams are generic projections of a knot into a plane in $\R^3$. The image of such a projection is a planar $4$-valent graph. Hence, the points of the projection plane splits into regions (faces), arcs (edges) and crossings (vertices) of the diagram. Diagrams of isotopic knots are connected by a sequence of local transformations called Reidemeister moves.

Below we try to 
generalize the way to assign a diagram to a knot.

Consider a family $\mathcal P\subset 2^{\R^3}$ of subsets of $\R^3$ such that $\bigcup_{p\in\mathcal P} p=\R^3$. We will call the elements $p\in\mathcal P$ \emph{probes}. Below we consider the case when probes are oriented lines.

Given a link $L$, it defines a stratification $\Pi(L)$ of $\mathcal P$. The strata are distinguished by the number of intersection points of a probe in the stratum with the link, and the equivalence classes of the germs of the intersection points. For a link $L$ in general position with respect to $\mathcal P$, we call the stratification $\Pi(L)$ the \emph{diagram} of $L$ in $\mathcal P$. The correspondence $L\mapsto\Pi(L)$ is called the \emph{projector} of $\mathcal P$.


Consider the set of oriented lines in $\R^3$ (identified with the tangent bundle $TS^2$) and the set of directions, i.e. oriented lines which contain the origin, (identified with $S^2$). Consider the projector map $\pi\colon S^2\times\R^3\to TS^2$ that assings to a direction $\nu$ and a point $x\in\R^3$ the unique oriented line $l$ which goes through $x$ in the direction $\nu$.

Choose a subset $\mathcal D\subset S^2$, we will call it a \emph{set of admissible directions}. Let $\mathcal P_{\mathcal D}$ be the preimage of $\mathcal D$ under the natural projection $TS^2\to S^2$.

For such a $\mathcal P_{\mathcal D}$, the diagram $\Pi(L)$ of a link $L$ is the image of the restriction $\pi|_{\mathcal D\times L}\colon \mathcal D\times L\to\mathcal P_{\mathcal D}$ stratified by singularity types of values.

Among the strata of $\Pi(L)$ we distinguish the \emph{regular} strata of codimension $(k-1)$ which consists of transversal $k$-secants, $k\ge0$, and \emph{singular} strata which lie in closures of regular strata.

\subsubsection{Spot (Reidemeister) projector}
Fix a direction $\nu\in S^2$ and consider $\mathcal D_\nu=\{\nu\}$. For a link $L$ in general position, the probe set $\mathcal P_\nu=\mathcal P_{\mathcal D_\nu}$ splits $\mathcal P_\nu=\mathcal P_\nu(L)_0\sqcup\mathcal P_\nu(L)_1\sqcup\mathcal P_\nu(L)_2$ where $\mathcal P_\nu(L)_i$ is the set of $i$-secants of $L$, $i=0,1,2$. Connected components of $\mathcal P_\nu(L)_0$ are called \emph{regions} of the diagram $\Pi(L)$,  components of $\mathcal P_\nu(L)_1$ are \emph{(semi) arcs} of the diagram, and elements of $\mathcal P_\nu(L)_2$ are called \emph{crossings}. The stratum $\mathcal P_\nu(L)_i$ has codimension $i$ in the $2$-dimensional space $\mathcal P_\nu$.

For a generic isotopy $L_t$, $t\in[0,1]$, the stratifications $\Pi(L_t)$ can include strata that correspond to Reidemeister moves: a tangent $1$-secant, a $2$-secant with zero sign, or $3$-secant (Fig.~\ref{fig:2secant_singularities})
\begin{figure}[h]
    \centering
    \includegraphics[width=0.5\linewidth]{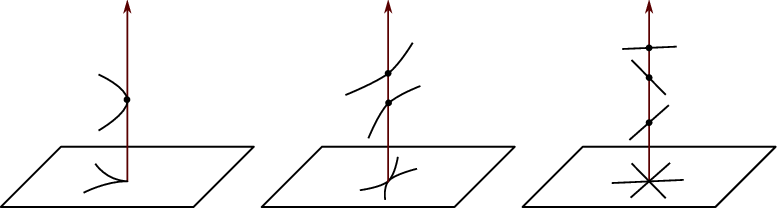}
    \caption{Secants corresponding to Reidemeister moves}
    \label{fig:2secant_singularities}
\end{figure}

\subsubsection{Spinner projector}

Fix a direction $\nu\in S^2$ and consider the set $\mathcal D^\bot_\nu=\{z\in S^2\mid z\cdot\nu=0\}$. The probe set $\mathcal P^\bot_\nu=\mathcal P_{\mathcal D^\bot_\nu}$ has dimension $3$.

One can look at the spinner projector as a $1$-parameter family of projections on planes parallel to $\nu$. This approach was exploited by Fiedler and Kurlin~\cite{FK2} who described diagram singularities.

Assume that $\nu$ is unit vector of the axis $Oz$. A link $L$ in general position has a finite number of extrema (minima and maxima) with respect to the height function $z$. The space $\mathcal P^\bot_\nu$ splits into the following strata according the codimension:
\begin{itemize}
    \item[$0$:] $0$-secants;
    \item[$1$:] transversal $1$-secants;
    \item[$2$:] transversal $2$-secants, transversal $1$-secants at extrema;
    \item[$3$:] non zero transversal $3$-secants, transversal $2$-secants on an extremum, tangent $1$-secants at extrema.
 \end{itemize}

For a generic isotopy $L_t$, $t\in[0,1]$, a diagram $L_t$ can contain two extrema of the same height, or an infliction point of the height function. Then the following new strata can appear:
\begin{itemize}
    \item transversal $1$-secants on an infliction point ($1$-parametric strata);
    \item tangent $1$-secants on an infliction point;
    \item transversal $2$-secants on two extrema;
    \item tangent $2$-secants on an extremum;
    \item transversal $3$-secants on an extremum;
    \item zero transversal $3$-secants;
    \item transversal $4$-secant.
\end{itemize}

Perturbations of singularities of the last four strata can include a trisecant. These singularities are shown in Fig.~\ref{fig:3secant_singularities}.
\begin{figure}[h]
    \centering
    \includegraphics[width=0.9\linewidth]{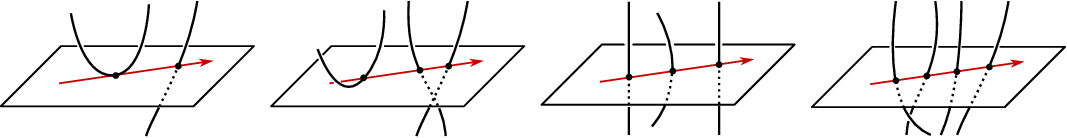}
    \caption{Singularities with a trisecant in a perturbation}
    \label{fig:3secant_singularities}
\end{figure}

\subsubsection{Omni projector}

Let $\mathcal D=S^2$ and $\mathcal P^o=\mathcal P_{S^2}$ be the set of all oriented lines in $\R^3$. The omni probe set $\mathcal P^o$ has dimension $4$.

Singular points of the projector map were described by David~\cite{Dav} who also proved a transversality theorem for it. Note that David considered not only orthogonal projections, but also radial projections.

A link $L$ in general position can have flattening points (equivalent to the singularity $(t,t^2,t^5)$) but not inflection points (equivalent to $(t,t^3,t^4)$). So the space of secants splits into the following strata according the codimension:
\begin{itemize}
    \item[$0$:] $0$-secants;
    \item[$1$:] transversal $1$-secants;
    \item[$2$:] transversal $2$-secants;
    \item[$3$:] transversal $3$-secants, tangent $1$-secant on a non-flattening point;
    \item[$4$:] non zero transversal $4$-secants, $2$-secant on a tangent point, tangent $1$-secant on a flattening point.
 \end{itemize}

For a generic isotopy $L_t$, $t\in[0,1]$, a diagram $L_t$ can contain inflection points and biflattening points (equivalent to $(t,t^2,t^7)$). Then the following new strata can appear:
\begin{itemize}
    \item tangent $1$-secants on a biflattening or infliction point;
    \item $2$-secants on a flattening point;
    \item $2$-secants with two tangent point;
    \item $3$-secant with a tangent point;
    \item zero transversal $4$-secants;
    \item transversal $5$-secant.
\end{itemize}

\section{Multicrossing complex}\label{sect:multicrossing_complex}

Let $L$ be a classical oriented link. We can think that $L$ lies inside a ball $B\subset\R^3$. Let $N(L)$ be a tubular neighborhood of $L$, and $M_L=\overline{B\setminus N(L)}$. Choose a point $x_0\in S=\partial B$.

For $n\ge 0$, consider the space $\widetilde{\mathcal C}(L,n)$ that consists of homotopy classes of tuples $\gamma=(\gamma_0,\dots,\gamma_n)$ of paths
\begin{gather*}
\gamma_0\colon (I,0,1)\to (M_L, S,\partial N(L)),\\
\gamma_i\colon (I,0,1)\to (M_L, \partial N(L),\partial N(L)),\ i=1,\dots,n-1,\\
\gamma_n\colon (I,0,1)\to (M_L, \partial N(L), S);
\end{gather*}
such that
\begin{itemize}
\item $\gamma_{i-1}(1)$ and $\gamma_{i}(0)$ are different points of the same meridian $\mu_i$, $1\le i\le n$.

Denote the part of the meridian $\mu_i$ from $\gamma_{i-1}(1)$ to $\gamma_{i}(0)$ by $\mu_{i,l}$, and the part from $\gamma_{i}(0)$ to $\gamma_{i-1}(1)$ by $\mu_{i,r}$.
\item the meridians $\mu_i$, $i=1,\dots,n$, are all distinct.
\end{itemize}

\begin{figure}[h]
    \centering
    \includegraphics[width=0.4\linewidth]{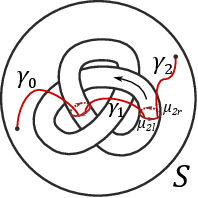}
    \caption{A multicrossing}
    \label{fig:gensecant}
\end{figure}

\begin{remark}
    1. Since $S$ is simply connected, we can suppose that $\gamma_0$ starts and $\gamma_n$ ends at $x_0\in S$.

    2. There is a map from $\mathcal S_n(L)$ to $\widetilde{\mathcal C}(L,n)$ such that for an $s\in\mathcal S_n(L)$ the paths $\gamma_0,\dots, \gamma_n$ are the connected components of $s\cap M_L$.
\end{remark}

The sequence $(\widetilde{\mathdutchcal C}(L,n))_{n\ge 0}$ with the maps $\partial^n_{i,l},\partial^n_{i,l}\colon \widetilde{\mathdutchcal C}(L,n)\to\widetilde{\mathdutchcal C}(L,n-1)$, $i=1,\dots,n$, given by the formulas
\begin{gather*}
    \partial^n_{i,l}(\gamma_0,\dots,\gamma_n)=(\gamma_0,\dots,\gamma_{i-2},\gamma_{i-1}\mu_{i,l}\gamma_i,\gamma_{i+1},\dots,\gamma_n),\\
    \partial^n_{i,r}(\gamma_0,\dots,\gamma_n)=(\gamma_0,\dots,\gamma_{i-2},\gamma_{i-1}(\mu_{i,r})^{-1}\gamma_i,\gamma_{i+1},\dots,\gamma_n),
\end{gather*}
is a semi-cubical set.

For $n\ge 2$, consider a subset $\widetilde{\mathdutchcal D}(L,n)=\bigcup_{i=1}^{n-1}\widetilde{\mathdutchcal D}_i(L,n) \subset \widetilde{\mathdutchcal C}(L,n)$ where $\widetilde{\mathdutchcal D}_i(L,n)$ consists of elements $\gamma$ for which
there exists a curve $\delta\colon[0,1]\to\partial N(L)$ such that $\delta(0)=\gamma_i(0)$, $\delta(1)=\gamma_i(1)$, $\delta\cap\bigcup_{i=1}^n\mu_i=\partial\delta$, and the loop $\gamma_i\delta^{-1}$ is contractible in $M_L$. In other words, $\gamma_i$ is homotopic to a curve $\delta$ in $\partial N$ which intersects the meridians $\mu_j$, $j=1,\dots,n$, only in the end points which are the endpoints of $\gamma_i$. In particular, this means that there is no other meridians between $\mu_i$ and $\mu_{i+1}$.


\begin{lemma}
    1. For any $1\le i<j\le n$ and $p,q\in\{l,r\}$, $d^{n-1}_{i,p}d^n_{j,q}=d^{n-1}_{j-1,q}d^n_{i,p}$;

    2. For any $\gamma\in \widetilde{\mathdutchcal D}_i(L,n)$ and any $p\in\{l,r\}$, $d^n_{i,p}(\gamma)=d^n_{i+1,p}(\gamma)$ and $d^n_{j,p}(\gamma)\in\widetilde{\mathdutchcal D}_{i-1}(L,n-1)$ when $j<i$, and $d^n_{j,p}(\gamma)\in\widetilde{\mathdutchcal D}_{i}(L,n-1)$ when $j>i+1$.
\end{lemma}

For an abelian group $A$, consider the complex
\[
    C_n(L,A)=\widetilde{\mathdutchcal C}(L,n)\otimes A/\widetilde{\mathdutchcal D}(L,n)\otimes A
\]
with the differential
\[
d(\gamma)=\sum_{i=1}^n(-1)^i(\partial^n_{i,l}(\gamma)-\partial^n_{i,r}(\gamma)),
\]
and consider the complex
\[
    C^n(L,A)=\{\phi\in Hom(\widetilde{\mathdutchcal C}(L,n),A)\mid \phi|_{\widetilde{\mathdutchcal D}(L,n)}=0\}
\]
with the differential $(d\phi)(\gamma)=\phi(d\gamma)$.

\begin{definition}\label{def:homotopy_crossing_homology}
    The homology $H_*(L,A)=H(C_*(L,A))$ and $H^*(L,A)=H(C^*(L,A))$ are called the \emph{homotopy multicrossing homology} and \emph{homotopy multicrossing cohomology} of the link $L$ with coefficients in $A$.
\end{definition}

By definition, the multicrossing homology is an invariant of oriented links.
\begin{theorem}[\cite{Nca}]
Multicrossing homologies of isotopic links are isomorphic.
\end{theorem}

\begin{remark}
    Similarly to the homotopy multicrossing homology of a link, one can define the \emph{isotopy multicrossing homology} $H^{iso}_*(L,A)$ and cohomology $H_{iso}^*(L,A)$ based on the space $\widetilde{\mathcal C}_{iso}(L,n)$ that consists of isotopy classes of tuples $\gamma=(\gamma_0,\dots,\gamma_n)$ of simple curves which satisfy the conditions above and the additional requirement:
\begin{itemize}
    \item the curve $\partial^1_{1,l}\circ\cdots\circ\partial^n_{1,l}(\gamma)$ is the long unknot in $B$.
\end{itemize}
The formula of the differential is the same as above, and the spaces $\widetilde{\mathdutchcal D}^{iso}_i(L,n)$ consist of tuples $\gamma$ for which
there exists a simple curve $\delta\colon[0,1]\to\partial N(L)$ such that $\delta(0)=\gamma_i(0)$, $\delta(1)=\gamma_i(1)$, $\delta\cap\bigcup_{i=1}^n\mu_i=\partial\delta$, and the loop $\gamma_i\delta^{-1}$ bounds a disk in $M_L$.
\end{remark}

\subsection{Positive and negative multicrossing homology}

Consider the inversion map $\gamma\mapsto \bar\gamma$, $\gamma\in \widetilde{\mathdutchcal C}(L,n)$, defined by the formula
\[
\overline{(\gamma_0,\dots,\gamma_n)}=(\gamma_n^{-1},\dots,\gamma^{-1}_0).
\]

From the definition of inversion we get the following properties.
\begin{lemma}\label{lem:multicrossing_inversion}
For any $\gamma\in \widetilde{\mathdutchcal C}(L,n)$
\begin{enumerate}
    \item $\bar{\bar\gamma}=\gamma$;
    \item for any $p\in\{l,r\}$ and $i\in\{1,\dots,n\}$,
$
\overline{\partial^n_{i,p}(\gamma)}=\partial^n_{n+1-i,\bar p}(\bar\gamma)
$,
where $\bar l=r, \bar r=l$;

    \item  $\overline{d(\gamma)}=(-1)^n d(\bar\gamma)$;

    \item if $\gamma\in \widetilde{\mathdutchcal D}_i(L,n)$, $i=1,\dots,n-1$, then $\bar\gamma\in\widetilde{\mathdutchcal D}_{n+2-i}(L,n)$.
\end{enumerate}
\end{lemma}

For $\gamma\in \widetilde{\mathdutchcal C}(L,n)$, denote $\iota(\gamma)=(-1)^{\frac{n(n+1)}2}\bar\gamma$. Lemma~\ref{lem:multicrossing_inversion} implies immediately the following statements.

\begin{corollary}\label{cor:multicrossing_inversion}
\begin{enumerate}
    \item $\iota^2=id$;
    \item $\iota d=d\iota$;
    \item $\iota(\widetilde{\mathdutchcal D}(L,n))=\widetilde{\mathdutchcal D}(L,n)$.
\end{enumerate}
\end{corollary}

Let $A$ be an abelian group. Then the map $\iota$ defines an involution of the chain complex $C_*(L,A)$. Denote $C^\pm_*(L,A)=C_*(L,A)/(id\mp\iota)$.

\begin{definition}\label{def:positive_multicrossing_homology}
    The homology $H^\pm_*(L,A)$ of the complex $C^\pm_*(L,A)$ is called the \emph{positive} (\emph{negative}) multicrossing homology of $L$ with coefficients in $A$.

    Analogously, one defines the positive multicrossing cohomology $H_+^*(L,A)$ and the negative multicrossing homology $H_-^*(L,A)$ of the link $L$.
\end{definition}

\begin{remark}
    There is a natural projection $H_*(L,A)\to H^\pm_*(L,A)$.
\end{remark}

\subsection{Secant class of knot}\label{subsect:secant_classes}

Fix a direction $\nu\in S^2$. Denote $\mathcal P_2=\mathcal P_\nu$, $\mathcal P_3=\mathcal P^\bot_\nu$, $\mathcal P_4=\mathcal P^o$.

Let $L$ be a link in $\R^3$. Define the sets of \emph{admissible $n$-secants}, $n=2,3,4$, by the formula $\mathcal S^{adm}_n(L)=\mathcal S_n(L)\cap\mathcal P_n$. For $n=2,3,4$, we say that $L$ is \emph{$n$-generic} if
\begin{itemize}
    \item all $n$-secants of $L$ in $\mathcal P_n$ are transversal;
    \item all $n$-secants in $\mathcal S^{adm}_n(L)$ have non zero signs;
    \item there are no $(n+1)$-secants of $L$ in $\mathcal P_n$.
\end{itemize}

\begin{definition}\label{def:n-secant_class}
    Let $n\in\{2,3,4\}$ and a link $L$ be $n$-generic. The sum
\[
\sigma_n(L)=\sum_{s\in \mathcal S^{adm}_n(L)}\sgn(s)\cdot s
\]
is called the \emph{$n$-secant class} of the link $L$. The sum $\sigma_n(L)$ is considered as an element of $H_2(L,\Z)$ for $n=2$, an element of $H_3^-(L,\Z)$ for $n=3$, and an element of $H_4^+(L,\Z)$ for $n=4$.
\end{definition}

Let us formulate the main theorem of the paper.

\begin{theorem}\label{thm:secant_class_invariance}
Let $L$ be an oriented link. Then
\begin{enumerate}
    \item $\sigma_2(L)$ is an invariant class in $H_2(L,\Z)$;
    \item $\sigma_3(L)$ is an invariant class in $H^-_3(L,\Z)$;
    \item $\sigma_4(L)$ is an invariant class in $H^+_4(L,\Z)$.
\end{enumerate}
The invariance of $\sigma_n(L)$ means that for any isotopy $f\colon L\to L'$,
\[
\sigma_n(L')=f_*(\sigma_n(L))\in H_n(L',\Z),\quad n=2,3,4.
\]
\end{theorem}

Theorem~\ref{thm:secant_class_invariance} will be proved in Section~\ref{sect:proof}.

\begin{remark}\label{rem:isotopic_secant_classes}
    The same formula~\eqref{eq:secant_class} defines invariant secant classes in the isotopy multicrossing homologies $H^{iso}_2(L,\Z)$, $H^{iso-}_3(L,\Z)$ and $H^{iso+}_4(L,\Z)$. These classes are universal for link invariants which represent the sum of labels of $n$-secants.
\end{remark}

\begin{corollary}\label{cor:secant_cocycle_invariant}
Let $A$ be an abelian group. Then
\begin{enumerate}
    \item for any multicrossing $2$-cocycle $\theta$ with coefficients in $A$ the value $I_\theta(L)=\theta(\sigma_2(L))$ is invariant;
    \item for any negative multicrossing $3$-cocycle $\theta$ with coefficients in $A$ the value $I_\theta(L)=\theta(\sigma_3(L))$ is invariant;
    \item for any positive multicrossing $4$-cocycle $\theta$ with coefficients in $A$ the value $I_\theta(L)=\theta(\sigma_4(L))$ is invariant.
\end{enumerate}
\end{corollary}

Let us see how the secant classes change under symmetries of the knot.

\begin{proposition}\label{prop:knot_symmetries_secant_class}
    Let $K$ be  an oriented knot, $-K$ the reversed knot, and $\bar K$ the mirror knot. Then
\begin{enumerate}
    \item the multicrossing complexes for $-K$ and $\bar K$ are identified with the complex $(C_*(K,A), -d)$;
    \item the maps $\tau_n=(-1)^n id\colon C_n(K,A)\to C_n(K,A)$ define an isomorphism $\tau_*$ between $H_*(K,A)$ and $H_*(-K,A)=H_*(\bar K,A)$;
    \item for $A=\Z$, $\tau_*(\sigma_n(K))=\sigma_n(-K)=\sigma_n(\bar K)$, $n=3,4$, but $\tau_*(\sigma_2(K))=\sigma_2(-K)=-\sigma_2(\bar K)$.
\end{enumerate}
\end{proposition}

\begin{proof}
    1. The set of $n$-secants of $-K$ coincides with those of $K$. Hence, one can identify the chain spaces $C_n(K,A)$ and $C_n(-K,A)$. But in the complex for $-K$, the maps $\partial_{i,l}$ and $\partial_{i,r}$ swap places. Hence, the differential $d$ becomes $-d$.

    One gets the knot $\bar K$ by changing the orientation of $\R^3$. The rest arguments coincide with those for $-K$.

    2. The second statement is an immediate consequence of the first statement.

    3. Let us look how the secant signs change after passing to $-K$. Since the expressions $\Delta_n$ are linear in $\dot\alpha_i$ (see Section~\ref{subsect:secant_sign}), the sign of an $n$-secant $s$ of $K$ considered as a secant of $-K$ is equal to $(-1)^n\sgn(s)$. Hence,
\[
\tau_*(\sigma_n(K))=(-1)^n\sum_{s\in \mathcal S^{adm}_n(K)}\sgn(s)s=
\sum_{s\in \mathcal S^{adm}_n(-K)}\sgn_{-K}(s)s=\sigma_n(-K),
\]
$n=2,3,4$.

Now, look at the signs of secants of $\bar K$. Changing the orientation of $\R^3$ preserves scalar products of vectors but changes the sign of triple products. By formulas for $\Delta_n$ in Section~\ref{subsect:secant_sign}, we get $\sgn_{\bar K}(s)=-\sgn_K(s)$ for $2$-secants and $3$-secants, and $\sgn_{\bar K}(s)=\sgn_K(s)$ for $4$-secants. This observation implies the third statement for $\bar K$.
\end{proof}

\subsection{Component cocycle invariants}\label{subsect:component_cocycles}

Let us give some examples of secant cocycle invariants. Consider a link $L=K_1\cup\cdots\cup K_m$ with $m$ enumerated component.

For an $n$-secant $s$ on arcs $\alpha_1(t),\dots,\alpha_n(t)$ of $L$, its \emph{component type} $c(s)$ is equal to $(i_1,\dots,i_n)$, $i_k\in\{1,\dots,m\}$, $k=1,\dots,n$, if $\alpha_k(t)\subset K_{i_k}$ for any $k=1,\dots,n$. Consider functions $\theta_{i_1\dots i_n}\colon\mathcal S_n(L)\to\Z$ such that
\[
\theta_{i_1\dots i_n}(s)=\left\{\begin{array}{cl}
    1, &  c(s)=(i_1,\dots,i_n)\\
    0, & \mbox{otherwise}.
\end{array}\right.
\]

\begin{proposition}\label{prop:component_secant_cocycles}
\begin{enumerate}
    \item Let $i_1\ne i_2$. Then $\theta_{i_1i_2}$ is a multicrossing $2$-cocycle.
    \item Let $i_1\ne i_2\ne i_3$. Then $\theta^-_{i_1 i_2 i_3}=\theta_{i_1i_2i_3}-\theta_{i_3i_2i_1}$ is a negative multicrossing $3$-cocycle.
    \item Let $i_1\ne i_2$. Then $\theta_{i_1i_2i_1}$ is a negative multicrossing $3$-cocycle with coefficients in $\Z_2$.
    \item Let $i_1\ne i_2\ne i_3\ne i_4$. Then $\theta^+_{i_1 i_2 i_3 i_4}=\theta_{i_1 i_2 i_3 i_4}+\theta_{i_4 i_3 i_2 i_1}$ is a positive multicrossing $4$-cocycle.
\end{enumerate}
\end{proposition}
\begin{proof}
Let us prove the third statement. Denote $\theta=\theta_{i_1i_2i_1}$. The condition $i_1\ne i_2$ implies that $\theta(s)=0$ for any $3$-secant $s\in\tilde{\mathdutchcal D}(L,3)$.

For any $4$-secant $s$ and $i\in\{1,2,3,4\}$, $\theta(\partial_{i,l}(s))=\theta(\partial_{i,r}(s))$. Hence $\theta(d(s))=0$, and $\theta$ is a cocycle.

For any $3$-secant $s$, the component type $c(-s)$ is reverse to $c(s)$. Since $(i_1,i_2,i_1)$ is palindromic, $\theta(-s)=\theta(s)\equiv-\theta(s)\pmod 2$. Hence, $\theta$ is a negative cocycle with coefficients in $\Z_2$.

The other statements are proved analogously.
\end{proof}

Denote $I^\epsilon_{i_1\dots i_n}=I_{\theta^\epsilon_{i_1\dots i_n}}$. By Corollary~\ref{cor:secant_cocycle_invariant} and Proposition~\ref{prop:component_secant_cocycles} the functions $I_{i_1i_2}$, $I^-_{i_1i_2i_3}$, $I_{i_1i_2i_1}\pmod 2$, $I^+_{i_1i_2i_3i_4}$, $i_1\ne i_2\ne i_3\ne i_4$, are link invariants. Let us calculate these invariants.

\begin{proposition}\label{prop:component_cocycle_ij}
 For any link $L$ and $i\ne j$,  $I_{ij}(L)=lk_{ij}(L)$ where $lk_{ij}(L)=lk(K_{i},K_{j})$ is the linking number of the $i$-th and $j$-th components.
\end{proposition}
\begin{proof}
    The admissible $2$-secants in $\mathcal S_2^{adm}(L)$ correspond to the crossings of $L$. Then by definition
\[
    I_{ij}(L)=\sum_{s\in\mathcal S^{adm}_2(L)}\sgn(s)\theta_{ij}(s)=
    \sum_{s: K_i{\scriptsize\mbox{ under }} K_j}\sgn(s)=lk(K_i,K_j).
\]
\end{proof}

\begin{proposition}\label{prop:component_cocycle_ijk}
 Let $i\ne j\ne k\ne i$. Then for any link $L$,
$I^-_{ijk}(L)=0$ and $I_{iji}(L)=lk_{ij}(L)\pmod 2$.
\end{proposition}

\begin{lemma}\label{lem:component_cocycle_ijk}
    Let links $L=K_1\cup\cdots\cup K_m$ and $L'$ obtained from $L$ by switching a crossing on components $K_p$ and $K_q$, $1\le p,q\le m$. Then for any $i\ne j\ne k\ne i$
\begin{enumerate}
    \item $I^-_{ijk}(L)=I^-_{ijk}(L')$;
    \item $I^-_{iji}(L)-I^-_{iji}(L')=\left\{\begin{array}{cl}
        1, & \{i,j\}=\{p,q\}, \\
        0, &  \mbox{otherwise}
    \end{array}\right.\pmod 2$.
\end{enumerate}
\end{lemma}

\begin{proof}
Lift the crossing as shown in Fig.~\ref{fig:placeholder3secant_clasp_move}. Then the top part of $L$ contains horizontal trisecants of component types $(p,q,p)$ and $(q,p,q)$ (as well as two trisecants of types $(p,p,q)$ and $(q,q,p)$ which are not drawn in the figure) but the top part of $L'$ does not have horizontal trisecants. Counting the trisecants, we get the claim of the lemma.
\begin{figure}[h]
    \centering
    \includegraphics[width=0.25\linewidth]{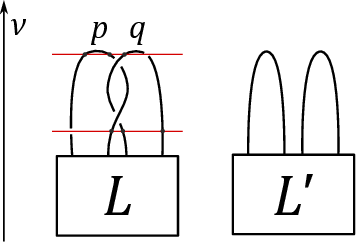}
    \caption{Links $L$ and $L'$}
    \label{fig:placeholder3secant_clasp_move}
\end{figure}
\end{proof}

\begin{proof}[Proof of Proposition~\ref{prop:component_cocycle_ijk}]
By Lemma~\ref{lem:component_cocycle_ijk}, the value $I^-_{ijk}(L)$ coincides with the value of $I^-_{ijk}$ on the unlink which is zero. The second statement of the lemma implies that the difference $I_{iji}(L)-lk_{ij}(L)\pmod 2$ does not change under crossing swiches. For the unlink $U$, $I_{iji}(L)=lk_{ij}(L)=0$, hence, $I_{iji}(L)=lk_{ij}(L)\pmod 2$ for any link $L$.
\end{proof}

\begin{proposition}[\cite{V}]\label{prop:component_cocycle_ijkl}
Let $i\ne j\ne k\ne l$. Then for any link $L$,
\[
I^+_{ijkl}(L)=lk_{ij}(L)\cdot lk_{kl}(L).
\]
\end{proposition}

\begin{lemma}\label{lem:component_cocycle_ijkl}
    Let links $L=K_1\cup\cdots\cup K_m$ and $L'$ obtained from $L$ by switching a crossing $c$ on components $K_p$ and $K_q$, $1\le p,q\le m$. Then for any $i\ne j\ne k\ne l$
\begin{equation}
   I^+_{ijkl}(L)-I^+_{ijkl}(L')=\left\{\begin{array}{cl}
        \sgn(c)lk_{kl}, & \{i,j\}=\{p,q\}, \{k,l\}\ne\{p,q\}, \\
        \sgn(c)lk_{ij}, & \{i,j\}\ne\{p,q\}, \{k,l\}=\{p,q\}, \\
        2\sgn(c)lk_{ij}-1, & \{i,j\}=\{k,l\}=\{p,q\}, \\
        0, &   \{i,j\}\ne\{p,q\}, \{k,l\}\ne\{p,q\}.
    \end{array}\right.
\end{equation}
\end{lemma}

\begin{proof}

Pull away the crossing $c$ from the rest part $R$ of the link as shown in Fig.~\ref{fig:4secant_skein_configuration}. Assume that the part $R$ lies near a plane orthogonal to a line $l$ connecting $c$ with $R$, and that the $2$-secant sheaf at $c$ cover $R$.
\begin{figure}[h]
    \centering
    \includegraphics[width=0.25\linewidth]{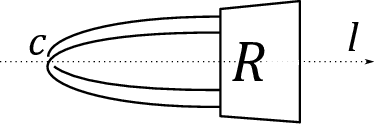}
    \caption{Link $L$}
    \label{fig:4secant_skein_configuration}
\end{figure}

Suppose that in $2$-secants close to  $l$ the first intersection point belongs to $K_p$. Diagrams $L$ and $L'$ differ by $4$-secants which pass near $c$ (or the corresponding crossing $c'$ in $L'$). These $4$-secants have component type $(p,q,r,s)$ for $L$ and $(q,p,r,s)$ for $L'$ and correspond to those crossings $\tilde c$ of the orthogonal projection of $R$ along $l$ where $K_s$ overcrosses $K_r$. By Remark~\ref{rem:4secant_geom} the sign of such a $4$-secant is equal to $\sgn(c)\sgn(\tilde c)$ (or $\sgn(c')\sgn(\tilde c)$).

Let $i=p$, $j=q$, and $\{k,l\}\ne\{p,q\}$.  The impact of the $4$-secants of $L$ or component type $(i,j,k,l)$ which pass near $c$ is
\[
\sum_{\tilde c\colon K_l{\scriptsize\mbox{ over }} K_k}\sgn(c)\sgn(\tilde c)=\sgn(c) lk_{kl}(L),
\]
and the impact of the $4$-secants of $L'$ or component type $(i,j,k,l)$ which pass near $c'$ is $0$. Hence,
\[
I^+_{ijkl}(L)-I^+_{ijkl}(L')=\sgn(c) lk_{kl}(L)-0=\sgn(c) lk_{kl}(L).
\]
Analogously, if $i=q$, $j=p$, and $\{k,l\}\ne\{p,q\}$ then
\[
I^+_{ijkl}(L)-I^+_{ijkl}(L')=0-\sgn(c') lk_{kl}(L')=\sgn(c) lk_{kl}(L)
\]
since $\sgn(c')=-\sgn(c)$ and $lk_{kl}(L')=lk_{kl}(L)$. This proves the first case.

Let $i=k=p$, $j=l=q$. For the component types $(i,j,i,j)$ and $(j,i,j,i)$ , the $4$-secants of $L$ which are close to $l$, is
\[
\sum_{\tilde c\in R\colon K_j{\scriptsize\mbox{ over }} K_i}\sgn(c)\sgn(\tilde c)=\sgn(c)(lk_{ij}(L)-\sgn(c))=\sgn(c)lk_{ij}(L)-1
\]
since we exclude $c$ when we count the crossings of $K_i$ and $K_j$ in $R$. The impact of $4$-secants of $L'$ is
\[
\sum_{\tilde c\in R\colon K_i{\scriptsize\mbox{ over }} K_j}\sgn(c)\sgn(\tilde c)=\sgn(c')(lk_{ij}(L')-\sgn(c'))=-\sgn(c)lk_{ij}(L)
\]
since $lk_{ij}(L')=lk_{ij}(L)+\sgn(c')$. Then
\[
I^+_{ijkl}(L)-I^+_{ijkl}(L')=2\sgn(c) lk_{ij}(L)-1.
\]
The case $i=k=q$, $j=l=p$ and other cases are proved analogously.
\end{proof}

\begin{proof}[Proof of Proposition~\ref{prop:component_cocycle_ijkl}]
By Lemma~\ref{lem:component_cocycle_ijkl} the difference $I^+_{ijkl}-lk_{ij}lk_{kl}$ does not change under crossing switches. For the unlink $U$, $I^+_{ijkl}(U)=lk_{ij}(U)=lk_{kl}(U)=0$. Hence, $I^+_{ijkl}(L)=lk_{ij}(L)lk_{kl}(L)$ for any link.
\end{proof}

\begin{example}[Alternating quadrisecants of long knots]\label{exa:long_knot_alternating_quadrisecant}
Fix a ball $B_0$ in $\R^3$ and an arc $\alpha\subset\overline{\R^3\setminus B_0}$ such that $\partial\alpha=\alpha\cap B_0$. A \emph{long knot} is a knot $K$ that lies inside $B_0$ except the arc $\alpha$, i.e. $K\setminus B_0=\alpha$. Two long knots $K_1$ and $K_2$ are isotopic if there is an isotopy $\phi$ of $\R^3$ such that $\phi|_{\R^3\setminus B_0}=id$ and $\phi(K_1)=K_2$.

By analogy with the knot case one can define a multicrossing complex and secant classes for long knots.

Given a long knot $K$, choose a point $P$ in the arc $\alpha$. For a secant $s$ that does not contain $P$, there are two orders on its intersection points: the one induced by the knot orientation and the starting point $P$; and the other induced by the orientation of the secant line. A $4$-secant is called \emph{alternating} if the intersection points enumerated by the knot orientation, lie on the secant in the order $3142$ or $2413$ (Fig.~\ref{fig:long_knot}).

\begin{figure}[h]
    \centering
    \includegraphics[width=0.25\linewidth]{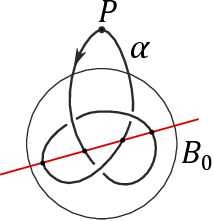}
    \caption{A long knot with an alternating $4$-secant}
    \label{fig:long_knot}
\end{figure}

Like component cocycles, the function $\theta_{alt}$ that assigns $1$ to alternating $4$-secants, and $0$ to the other $4$-secants, is a positive multicrossing cocycle. Hence, $\theta_{alt}$ defines an integer-valued invariant $I_{\theta_{alt}}$ on the long knots.

Budney, Conant, Scannell, and Sinha~\cite{BCSS} proved that the invariant $I_{\theta_{alt}}$ coincides with Vassiliev invariant $v_2$ of order $2$ (the $z^2$-coefficient of the Conway polynomial).
\end{example}

\section{Proof of the main theorem}\label{sect:proof}

\subsection{2-secants}\label{subsect:2s_proof}

For completeness, let us give the proof of invariance of $\sigma_2(L)$ (see~\cite[Proposition 59]{Nca}). We assume that the direction $\nu$ is the unit vector of the axis $Oz$.

Let us show first that $\sigma_2(L)$ is a cycle in the multicrossing complex. The elements of $\mathcal S^{adm}_2(L)$ correspond to the crossings of the diagram of $L$ (orthogonal projection of $L$ into the  plane $Oxy$). Given a $2$-secant $s\in\mathcal S^{adm}_2(L)$, the differentials $\partial_{i,q}(s)$, $i\in\{1,2\}$, $q\in\{l,r\}$ are vertical $1$-secants of the arcs incident to the crossing $s$ (Fig.~\ref{fig:s2_boundary}).
\begin{figure}[h]
    \centering
    \includegraphics[width=0.4\linewidth]{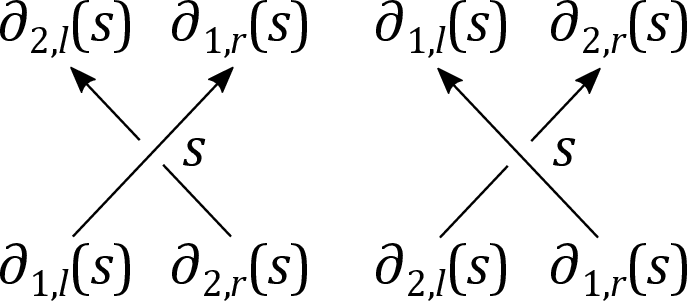}
    \caption{Differentials of a $2$-secant}
    \label{fig:s2_boundary}
\end{figure}

We see that $\sgn(s)d(s)$ is the sum of $1$-secant of outcoming arcs minus the sum of $1$-secants of the incoming arcs. Since each arc goes from one crossing to another, the summands in the sum $d\sigma_2(L)$ annihilates:
\[
d\sigma_2(L)=\sum_{s\in\mathcal S^{adm}_2(L)}\sgn(s)d(s)=0.
\]
Thus, $\sigma_2(L)$ is a cycle.

Now, we check the invariance of $\sigma_2(L)$. Let $f\colon L\to L'$ be an isotopy of links. If $f$ does not change the projection of $L$ into $Oxy$ or induces an isotopy of  link diagrams then there is a bijection between the crossings of the diagrams of $L$ and $L'$ and the corresponding $2$-secants of the crossings are isotopic.

If $f$ is an increasing first Reidemeister move and $s$ is the new vertical $2$-secant, then
\[
\sigma_2(L')-f_*(\sigma_2(L))=\sgn(s)s= 0\in C_2(L)
\]
since $s\in\tilde{\mathdutchcal D}(L,2)$.

If $f$ is an increasing second Reidemeister move and $s_1$ and $s_2$ are the new vertical $2$-secants then $s_1$ and $s_2$ are isotopic but $\sgn(s_1)=-\sgn(s_2)$. Hence,
\[
\sigma_2(L')-f_*(\sigma_2(L))=\sgn(s_1)s_1+\sgn(s_2)s_2=0.
\]

If $f$ is a third Reidemeister move then $\sigma_2(L')-f_*(\sigma_2(L))=\pm d(s)$ where $s$ is the vertical $3$-secant which appears during the move (Fig.~\ref{fig:s2_R3_invariance}).
\begin{figure}[h]
    \centering
    \includegraphics[width=0.7\linewidth]{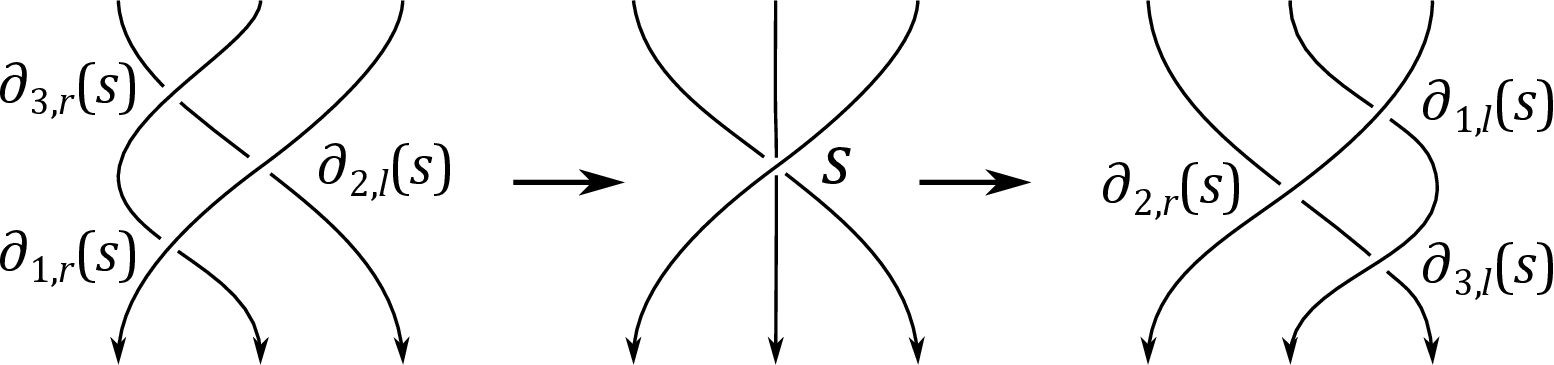}
    \caption{A third Reidemeister move}
    \label{fig:s2_R3_invariance}
\end{figure}

\subsection{3-secants}\label{subsect:3s_proof}

Assume that the direction $\nu$ is the unit vector of the axis $Oz$.

\begin{lemma}\label{lem:3secant_cocycle}
The chain $\sigma_3(L)$ is a cycle in $C^-_3(L,\Z)$.
\end{lemma}
\begin{proof}
For a general $z_0\in\R$, the intersection $L\cap\{z=z_0\}$ is transversal. Denote
\[
L\cap\{z=z_0\}=\{\alpha_1(z_0),\dots,\alpha_n(z_0)\}.
\]
Consider the chain
\[
\delta(z_0)=\sum_{1\le i<j\le n}\epsilon'(s_{ij})s_{ij}
\]
where $s_{ij}$ is the horizontal $2$-secant on $\alpha_i(z_0)$ and $\alpha_j(z_0)$, and
\[
\epsilon'(s_{ij})=(\dot\alpha_i(z_0),\nu)\cdot(\dot\alpha_j(z_0),\nu).
\]
Since $\epsilon'(s_{ij})=\epsilon'(s_{ji})$, $\delta(z_0)$ as a chain in $C_2^-(L)$ does not depend on the enumeration of $L\cap\{z=z_0\}$.

For $z_0\ll 0$ or $z_0\gg 0$, $L\cap\{z=z_0\}=\emptyset$ and $\delta(z_0)=0$. We will write this as $\delta(\pm\infty)=0$. Let us look how $\delta(z_0)$ changes when $z_0$ grows.

If a layer $\R^2\times[z_1,z_2]$ does not contain extremum points of $L$ and horizontal trisecants then $\delta(z_1)$ and $\delta(z_2)$ are sums of pairwise isotopic $2$-secants, hence, $\delta(z_1)=\delta(z_2)$.

If $\{z=z_0\}$ contains a local minimum of $L$ then for a small $\epsilon>0$
\begin{gather*}
L\cap\{z=z_0-\epsilon\}=\{\alpha_1(z_0-\epsilon),\dots,\alpha_n(z_0-\epsilon)\},\\
L\cap\{z=z_0+\epsilon\}=\{\alpha_1(z_0+\epsilon),\dots,\alpha_n(z_0+\epsilon),\alpha_{n+1}(z_0+\epsilon),\alpha_{n+2}(z_0+\epsilon)\}.
\end{gather*}
The $2$-secants $s_{ij}(z_0-\epsilon)$ and  $s_{ij}(z_0+\epsilon)$, $1\le i<j\le n$, are isotopic. We will omit the argument $z_0\pm\epsilon$. The new $2$-secants $s_{i,n+1}$ and $s_{i,n+2}$ are isotopic but $\epsilon'(s_{i,n+1})=-\epsilon'(s_{i,n+2})$ (Fig.~\ref{fig:3secant_boundary_minimum}). Then
\begin{multline*}
    \delta(z_0+\epsilon)- \delta(z_0-\epsilon)=\sum_{i=1}^n(\epsilon'(s_{i,n+1})s_{i,n+1}+\epsilon'(s_{i,n+2})s_{i,n+2})+\\
\epsilon'(s_{n+1,n+2})s_{n+1,n+2}=\epsilon'(s_{n+1,n+2})s_{n+1,n+2}=0
\end{multline*}
since $s_{n+1,n+2}\in\tilde{\mathdutchcal D}(L,2)$. Thus, $\delta(z_0+\epsilon)=\delta(z_0-\epsilon)$.

\begin{figure}[h]
    \centering
    \includegraphics[width=0.5\linewidth]{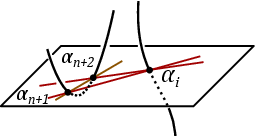}
    \caption{Secants near a mimimum point}
    \label{fig:3secant_boundary_minimum}
\end{figure}

Analogously, $\delta(z_0+\epsilon)=\delta(z_0-\epsilon)$ when $z_0$ is the value of a local maximum of $L$.

Let the plain $z=z_0$ contains a horizontal trisecant $s$ on $\alpha_i(z)$, $\alpha_j(z)$ and $\alpha_k(z)$. Then
\begin{multline*}
\delta(z_0+\epsilon)-\delta(z_0-\epsilon)=\epsilon'(s_{ij})s_{ij}(z_0+\epsilon)+\epsilon'(s_{ik})s_{ik}(z_0+\epsilon)+\epsilon'(s_{jk})s_{jk}(z_0+\epsilon)-\\
\epsilon'(s_{ij})s_{ij}(z_0-\epsilon)-\epsilon'(s_{ik})s_{ik}(z_0-\epsilon)-\epsilon'(s_{jk})s_{jk}(z_0-\epsilon).
\end{multline*}
From Proposition~\ref{prop:3secant_sign_geometry}, this sum is equal to $-\sgn(s)d(s)$.
\begin{figure}[h]
    \centering
    \includegraphics[width=0.5\linewidth]{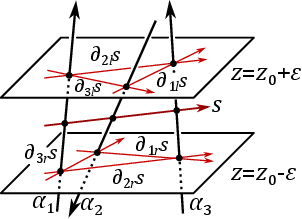}
    \caption{Secants near a negative $3$-secant}
    \label{fig:3secant_boundary}
\end{figure}
For example, in the configuration in Fig.~\ref{fig:3secant_boundary}, we have $i=1,j=2,k=3$, and $\sgn(s)=\epsilon'(s_{12})=\epsilon'(s_{23})=-1$, $\epsilon'(s_{13})=1$. Then
\begin{multline*}
\delta(z_0+\epsilon)-\delta(z_0-\epsilon)=\epsilon'(s_{12})s_{12}(z_0+\epsilon)+\epsilon'(s_{13})s_{13}(z_0+\epsilon)+\epsilon'(s_{23})s_{23}(z_0+\epsilon)-\\
\epsilon'(s_{12})s_{12}(z_0-\epsilon)-\epsilon'(s_{13})s_{13}(z_0-\epsilon)-\epsilon'(s_{23})s_{23}(z_0-\epsilon)=\\
-\partial_{3l}(s)+\partial_{2l}(s)-\partial_{1l}(s)+\partial_{3r}(s)-\partial_{2r}(s)+\partial_{1r}(s)=d(s).
\end{multline*}
The other cases are checked analogously.

Thus, the chain $\delta(z_0)$ changes when it passes a trisecant level. Then
\[
    0=\delta(+\infty)-\delta(-\infty)=\sum_{s\in\mathcal S^{adm}_3(L)}-\sgn(s)d(s)=-d(\sigma_3(s)).
\]
Thus, $\sigma_3(L)$ is a cycle.
\end{proof}

Now, we prove the invariance of $\sigma_3(L)$. Let $f_t\colon\R^3\to\R^3$, $t\in[0,1]$, be a generic isotopy between links $L$ and $L'$. If the links $L_t=f_t(L)$ do not have singular secants whose perturbations include $3$-secants, then the horizontal $3$-secants in $L_t$ form continuous $1$-parametric families, i.e. isotopies of the $3$-secants. Then $(f_1)_*(\sigma_3(L))=\sigma_3(L')$.

Look what happens at singular secants. Let $L_{t_0}$, $t_0\in(0,1)$, be a link with a singular secant. Choose a small $\epsilon>0$ and denote $g=f_{t_0+\epsilon}\circ f_{t_0-\epsilon}^{-1}$ for the isotopy between $L_{t_0-\epsilon}$ and $L_{t_0+\epsilon}$.

If $L_{t_0}$ contains a transversal $3$-secant on an extremum (Fig.~\ref{fig:tangent_min_3s_resolution}), we have an isotopy of the trisecant. Hence, $\sigma_3(L_{t_0+\epsilon})=g_*(\sigma_3(L_{t_0+\epsilon}))$.
\begin{figure}[h]
    \centering
    \includegraphics[width=0.8\linewidth]{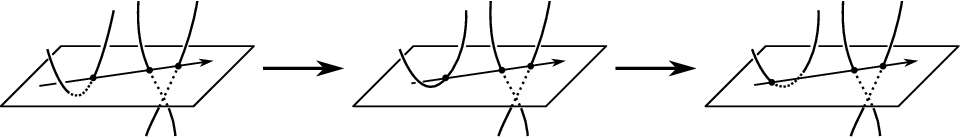}
    \caption{A transversal $3$-secant on an extremum}
    \label{fig:minimum_3secant_resolution}
\end{figure}

If $L_{t_0}$ contains a tangent $2$-secant (Fig.~\ref{fig:tangent_min_3s_resolution}), $\sigma_3(L_{t_0+\epsilon})$ differs from $g_*(\sigma_3(L_{t_0-\epsilon}))$ by a $3$-secant obtained by perturbation of the tangent $2$-secant. But this $3$-secant belongs to $\tilde{\mathdutchcal D}(L_{t_0+\epsilon},3)$, hence, $\sigma_3(L_{t_0+\epsilon})=g_*(\sigma_3(L_{t_0+\epsilon}))$.

\begin{figure}[h]
    \centering
    \includegraphics[width=0.8\linewidth]{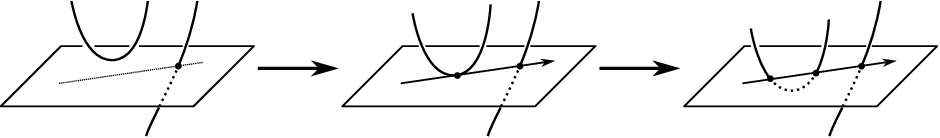}
    \caption{A tangent $2$-secant}
    \label{fig:tangent_min_3s_resolution}
\end{figure}

If $L_{t_0}$ contains a zero $3$-secant (Fig.~\ref{fig:zero_3s_resolution}), then it produces two isotopic $3$-secants with different signs which annihilate in the sum for $\sigma_3(L_{t_0+\epsilon})$. Hence, $\sigma_3(L_{t_0+\epsilon})=g_*(\sigma_3(L_{t_0-\epsilon}))$.
\begin{figure}[h]
    \centering
    \includegraphics[width=0.8\linewidth]{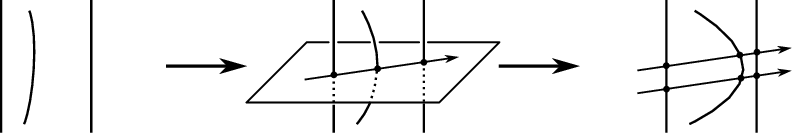}
    \caption{A zero $3$-secant}
    \label{fig:zero_3s_resolution}
\end{figure}

If $L_{t_0}$ contains a $4$-secant on $\alpha_i(z)=(p_i(z),z)$, $i=1,2,3,4$. A neighborhood of the $4$-secant is a graph of the movement of four points $p_i(z)$ in the plane (Fig.~\ref{fig:4secant_resolution}).

\begin{figure}[h]
    \centering
    \includegraphics[width=0.8\linewidth]{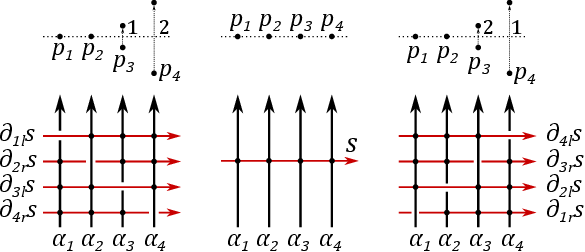}
    \caption{A $4$-secant and its perturbations}
    \label{fig:4secant_resolution}
\end{figure}

Assume that in $L_{t_0-\epsilon}$ the point $p_3$ moves early than $p_4$ (Fig.~\ref{fig:4secant_resolution} left), and in $L_{t_0+\epsilon}$ $p_3$ moves later than $p_4$ (Fig.~\ref{fig:4secant_resolution} right). Then the sum $\sigma_3(L_{t_0-\epsilon})$ includes positive $3$-secants $\partial_{1l}s, \partial_{2r}s, \partial_{3l}s, \partial_{4r}s$, and $\sigma_3(L_{t_0+\epsilon})$ includes positive $3$-secants $\partial_{1r}s, \partial_{2l}s, \partial_{3r}s, \partial_{4l}s$. Hence,
\begin{multline*}
    \sigma_3(L_{t_0+\epsilon})-g_*(\sigma_3(L_{t_0-\epsilon}))=\partial_{1r}s+ \partial_{2l}s+\partial_{3r}s+\partial_{4l}s-\\(\partial_{1l}s+\partial_{2r}s+\partial_{3l}s+\partial_{4r}s)=ds.
\end{multline*}

Thus, the secant class $\sigma_3$ is invariant when passing a singular value that implies the invariance in general case.

\subsection{4-secants}\label{subsect:4s_proof}

\begin{lemma}\label{lem:4secant_cocycle}
The chain $\sigma_4(L)$ is a cycle in $C^+_4(L,\Z)$.
\end{lemma}

\begin{proof}
    Let $L$ be a link in general position. For a point $p\in L$ consider the set $\mathcal S_3^{mid}(L,p)$ of transversal $3$-secants of $L$ for which $p$ is the middle intersection point. For each $s\in\mathcal S_3^{mid}(L,p)$ on arcs $\alpha_1(t),\alpha_2(t),\alpha_3(t)$, denote $\epsilon'(s)=\sgn(\alpha_1,\alpha_3)$. Consider the multicrossing chain defined by the formula
\[
\delta(p)=\sum_{s\in\mathcal S_3^{mid}(L,p)}\epsilon'(s)s.
\]
Since $\epsilon'(-s)=\sgn(s)$, $\delta(p)\in C^-_3(L,\Z)$.
Move the point $p$ along the oriented link and look how $\delta(p)$ changes.

Consider the central projection on a sphere $S$ with the center $p$. For a generic $p$ the projection of $L$ is a $4$-valent graph with two antipodal vertices of valency one that correspond to the tangent line in $p$. For a point $q\ne p$, besides its central projection in $S$ we consider its antipodal projection. Then the elements of $\mathcal S_3^{mid}(L,p)$ correspond to intersections of an arc with an antipodal arc of the projection of $L$. We call these intersections antipodal crossings.

During movement of $p$, the following transformations which change antipodal points can occur: $p$ passes over a flattening point, an arc passes over the tangent direction of $p$, a second or a third antipodal Reidemeister moves (Fig.~\ref{fig:middle_3secant_moves}).
\begin{figure}[h]
    \centering
    \includegraphics[width=0.8\linewidth]{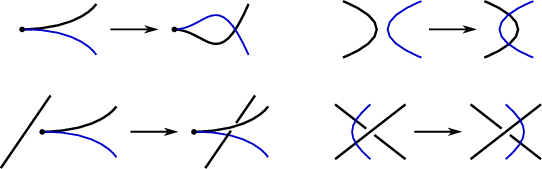}
    \caption{Moves of the central projection. Antipodal arcs are displayed in blue color.}
    \label{fig:middle_3secant_moves}
\end{figure}

The firs two transformations (Fig.~\ref{fig:middle_3secant_moves} left) produce a $3$-secant with an intersection point close to $p$. Then this trisecant belongs to $\tilde{\mathdutchcal D}(L_{t_0+\epsilon},3)$ and does not affect $\delta(p)$.

An antipodal second Reidemeister move produces two isotopic $3$-secants with opposite signs which annihilate and do not change $\delta(p)$.

An antipodal third Reidemeister move generates a triple point in the movement that corresponds a $4$-secant of $L$. Let’s examine this case more carefully.

Let $s$ be a $4$-secant of $L$ on arcs $\alpha_i(t)$, $i=1,2,3,4$. Since $L$ is generic, we suppose that the sign of $s$ is not zero, and the triples $\dot\alpha_i$, $\dot\alpha_j$, $\kappa$, $i<j$, are independent. Here $\kappa=\frac{\alpha_2-\alpha_1}{|\alpha_2-\alpha_1|}$ is the direction of the $4$-secant.

The $4$-secant $s$ appears as an antipodal third Reidemeister move for points that moves along arcs $\alpha_2$ and $\alpha_3$. For a small $\epsilon>0$, denote $p_\pm=\alpha_2(\pm\epsilon)$ and $q_\pm=\alpha_3(\pm\epsilon)$. We will calculate the difference $\delta(p_+)-\delta(p_-)$.

Assume that $\sgn(s)>0$ and $\sgn(\alpha_1,\alpha_2)>0$. Like in section~\ref{subsect:secant_sign}, consider the projection of trisecants in a plane $\Pi$ transversal to $\kappa$ (Fig~\ref{fig:s4_boundary}). By Proposition~\ref{prop:4secant_sign_geometry} we have $or(\alpha_{12\underline{3}},\alpha_{12\underline{4}})<0$.
\begin{figure}[h]
    \centering
    \includegraphics[width=0.3\linewidth]{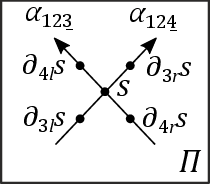}
    \caption{Projection of $3$-secants. The vector $\kappa$ is pointed towards the reader.}
    \label{fig:s4_boundary}
\end{figure}
The crossings is the intersection point of $s$ with $\Pi$, and the $3$-secants correspond to the incident edges. By definition of the multicrossing complex, these $3$-secants are differentials $\partial_{i,u} s$, $i\in\{3,4\}$, $u\in\{l,r\}$, of $s$.

On the other hand, the trisecants are the summands which differ $\delta(p_+)$ and $\delta(p_-)$. In order to see to which point a trisecant is assigned, reorient the curves $\alpha_{12\underline{3}},\alpha_{12\underline{4}}$ according the parameter of $\alpha_2(t)$. Then the incoming edges correspond to $\delta(p_-)$ and the outcoming ones correspond to $\delta(p_+)$.
By Remark~\ref{rem:4secant_geom} change of the orientation of $\alpha_{12\underline{i}}$ is determined by the sign
\[
\sgn(\alpha_1,\alpha_2)\cdot\sgn(\alpha_1,\alpha_i)=\sgn(\alpha_1,\alpha_i),\quad i=3,4.
\]
Then
\begin{multline*}
   \delta(p_+)-\delta(p_-)= \sgn(\alpha_1,\alpha_3)(\sgn(\alpha_1,\alpha_3)\partial_{4l}s-\sgn(\alpha_1,\alpha_3)\partial_{4r}s)+\\
   \sgn(\alpha_1,\alpha_4)(\sgn(\alpha_1,\alpha_4)\partial_{3r}s-\sgn(\alpha_1,\alpha_4)\partial_{3l}s)=\\
   -\partial_{3l}s+\partial_{3r}s+\partial_{4l}s-\partial_{4r}s.
\end{multline*}
Considering $4$-secant $-s$, we get an analogous formula for the points $q_\pm$:
\[
\delta(q_+)-\delta(q_-)=-\partial_{3l}(-s)+\partial_{3r}(-s)+\partial_{4l}(-s)-\partial_{4r}(-s)=-\partial_{2r}s+\partial_{2l}s+\partial_{1r}s-\partial_{1l}s.
\]
Thus,
\[
\sgn(s)ds=ds=\delta(p_+)-\delta(p_-)+\delta(q_+)-\delta(q_-).
\]
Other configurations of the $4$-secant $s$ give the same relation.

Now, choose a point $p_i$ in each component $K_i$, $i=1,\dots,m$, of the link $L$, pull $p_i$ along the component and look at the full change in the chains $\delta(p_i)$. Since $\delta(p_i)$ changes only at $4$-secants, we have
\[
0=\sum_{i=1}^m(\delta(p_i)-\delta(p_i))=\sum_{s\in\mathcal S_4^{adm}(L)}\sgn(s)ds=d(\sigma_4(L)).
\]
Thus, $\sigma_4(L)$ is a cycle.
\end{proof}

In order to prove invariance of $\sigma_4(L)$, consider a generic isotopy $f$ from a link $L$ to a link $L'$. We can only focus on cases when the isotopy passes a singularity involving a $4$-secant.

If $f$ passes a tangent $k$-secant, $k<4$, then any $4$-secant which can appear by a perturbation, has a pair of close intersection points, i.e. it belongs to $\tilde{\mathdutchcal D}(L,4)$. Hence, $f$ does not change the secant cycle.

Let $f$ passes a zero $4$-secant on arcs $\alpha_i(t)$, $i=1,2,3,4$. We can assume that $\sgn(\alpha_1,\alpha_2)\ne 0$. Then the projection of a neighbourhood of the $4$-secant in a plane transversal to it, looks like in Fig.~\ref{fig:4secant_zero}.
\begin{figure}[h]
    \centering
    \includegraphics[width=0.6\linewidth]{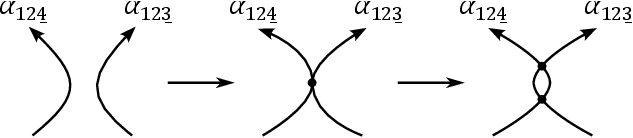}
    \caption{Projection of a zero $4$-secant}
    \label{fig:4secant_zero}
\end{figure}
Then a perturbation of the zero $4$-secant either has no $4$-secants or produces a pair of isotopic $4$-secants of opposite signs. These $4$-secants annihilate and do not affect $\sigma_4(L)$.

Let $f$ passes a $5$-secant $s$ on arcs $\alpha_i(t)$, $i=1,2,3,4,5$. We can suppose that the signs of the secants $\partial_{i,u}s$, $1\le i\le 5$, $u\in\{l,r\}$, are not zero, and that the vectors $\dot\alpha_i$, $\dot\alpha_j$, $\kappa$, $i<j$, are independent.

The projection of the isotopy $f$ in the plane $\Pi$ transversal to $s$ looks like a third Reidemeister move (Fig.~\ref{fig:5secant_boundary}).
\begin{figure}[h]
    \centering
    \includegraphics[width=0.7\linewidth]{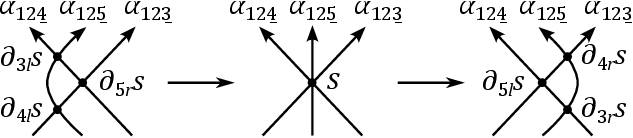}
    \caption{Projection of an isotopy at a $5$-secant}
    \label{fig:5secant_boundary}
\end{figure}
The crossings in the left and the right diagrams correspond to $4$-secants in $L$ and $L'$. By Proposition~\ref{prop:4secant_sign_geometry} and Remark~\ref{rem:4secant_geom}, for the configuration shown in Fig.~\ref{fig:5secant_boundary} we have
\begin{gather*}
\epsilon(\partial_{3l}s)=\epsilon(\partial_{3r}s)=\sgn(\alpha_1,\alpha_2),\\
\epsilon(\partial_{4l}s)=\epsilon(\partial_{4r}s)=\epsilon(\partial_{5l}s)=\epsilon(\partial_{5r}s)=-\sgn(\alpha_1,\alpha_2).
\end{gather*}

In order to catch the $4$-secants which do not intersect $\alpha_1(t)$ or $\alpha_2$, consider the $5$-secant $-s$ on the arcs $\alpha'_i(t)$, $1\le i\le 5$, where $\alpha'_i(t)=\alpha_{6-i}(t)$. Assume that the projection for $-s$ looks like as in Fig.~\ref{fig:5secant_boundary1} (we used there the relations $\partial_{i,u}(-s)=\partial_{6-i,\bar u}s$). Note that we should have the $4$-secant $\partial_{3r}(-s)=\partial_{3l}s$ in the left diagram (i.e. before the $5$-secant).
\begin{figure}[h]
    \centering
    \includegraphics[width=0.7\linewidth]{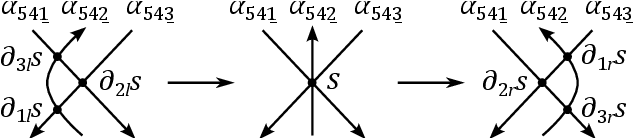}
    \caption{Projection of the isotopy at $-s$. }
    \label{fig:5secant_boundary1}
\end{figure}
By Proposition~\ref{prop:4secant_sign_geometry}
\begin{gather*}
\epsilon(\partial_{1l}s)=\epsilon(\partial_{1r}s)=\epsilon(\partial_{3l}s)=\epsilon(\partial_{3r}s)=\sgn(\alpha_4,\alpha_5),\\
\epsilon(\partial_{2l}s)=\epsilon(\partial_{2r}s)=-\sgn(\alpha_4,\alpha_5).
\end{gather*}

The relations on $\epsilon(\partial_{3l}s)$ imply $\sgn(\alpha_1,\alpha_2)=\sgn(\alpha_4,\alpha_5)$. Then
\begin{multline*}
\sigma_4(L')-f_*(\sigma_4(L))= \\(\epsilon(\partial_{1r}s)\partial_{1r}s+\epsilon(\partial_{2r}s)\partial_{2r}s+\epsilon(\partial_{3r}s)\partial_{3r}s+\epsilon(\partial_{4r}s)\partial_{4r}s+\epsilon(\partial_{5l}s)\partial_{5l}s)-\\(\epsilon(\partial_{1l}s)\partial_{1l}s+\epsilon(\partial_{2l}s)\partial_{2l}s+\epsilon(\partial_{3l}s)\partial_{3l}s+\epsilon(\partial_{4l}s)\partial_{4l}s+\epsilon(\partial_{5r}s)\partial_{5r}s)=\\
\sgn(\alpha_1,\alpha_2)(-\partial_{1l}s+\partial_{1r}s+\partial_{2l}s-\partial_{2r}s-\partial_{3l}s+\partial_{3r}s+\partial_{4l}s-\partial_{4r}s-\partial_{5l}s+\partial_{1r}s)=\\
\sgn(\alpha_1,\alpha_2)ds.
\end{multline*}
Thus, $\sigma_4(L')=f_*(\sigma_4(L))\in H^+_4(L',\Z)$. The other configurations of the $5$-secant are checked analogously.


\section{$(m,l)$-invariants}\label{sect:ml_invariants}

The multicrossing complex is infinite dimensional, therefore, to get a computable invariant, we use the idea of quandle colorings and consider maps to finite‑dimensional complexes, which will give us a series of knot invariants.

\subsection{$(m,l)$-homology of a group}

Denote the permutation group on $n$ elements by $\Sigma_n$.

Let $G$ be a group and $m,l\in G$ two commuting elements. For $n\ge 0$, consider the quotient sets $\tilde{\mathcal C}_n(G,m,l)$ of the product $\Sigma_n\times G^{\times (n+1)}$ by the relations:
\begin{gather*}
    (\sigma,g_0,\dots, g_{i-1}m,g_i,\dots,g_n)\sim  (\sigma,g_0,\dots, g_{i-1},mg_i,\dots,g_n),\quad i=1,\dots,n,\\
(\sigma(1)\sigma(2)\cdots\sigma(n),g_0,\dots,g_n)\sim\\
 (\sigma(2)\cdots\sigma(n)\sigma(1),g_0,\dots,g_{\sigma(1)-1},g_{\sigma(1)}l^{-1},lg_{\sigma(1)+1},g_{\sigma(1)+2},\dots, g_n).
\end{gather*}
Let ${\mathcal D}_{n,i}(G,m,l)$, $1\le i\le n-1$, be the subset of $\tilde{\mathcal C}_n(G,m,l)$ generated by the elements
\[
(\sigma,g_0,\dots,g_{i-1}, 1, g_{i+1},\dots,g_n)\mbox{ such that } |\sigma^{-1}(i)-\sigma^{-1}(i+1)|=1.
\]
Denote ${\mathcal D}_{n}(G,m,l)=\bigcup_{i=1}^{n-1}{\mathcal D}_{n,i}(G,m,l)$.

Consider the maps $\partial^n_{i,l}, \partial^n_{i,l}\colon \tilde{\mathcal C}_n(G,m,l)\to \tilde{\mathcal C}_{n-1}(G,m,l)$ defined by the formulas
\begin{gather*}
    \partial^n_{i,l}(\sigma,g_0,\dots,,g_n)=(\partial_i\sigma,g_0,\dots,g_{i-2}, g_{i-1}g_i, g_{i+1},\dots,g_n),\\
    \partial^n_{i,r}(\sigma,g_0,\dots,,g_n)=(\partial_i\sigma,g_0,\dots,g_{i-2}, g_{i-1}m^{-1}g_i, g_{i+1},\dots,g_n),
\end{gather*}
where
\[
(\partial_i\sigma)(k)=\left\{\begin{array}{cl}
    \sigma(k), & \sigma(k)<i\mbox{ and }k<\sigma^{-1}(i), \\
    \sigma(k)-1, & \sigma(k)> i\mbox{ and }k<\sigma^{-1}(i), \\
    \sigma(k+1), & \sigma(k+1)<i\mbox{ and }k\ge \sigma^{-1}(i), \\
    \sigma(k+1)-1, & \sigma(k+1)> i\mbox{ and }k\ge \sigma^{-1}(i),
\end{array}\right.\quad 1\le k\le n-1.
\]

For an abelian group $A$, consider the quotient complex
\[
    C_n(G,m,l;A)=\widetilde{\mathcal C}_n(G,m,l)\otimes A/{\mathcal D}_n(G,m,l)\otimes A
\]
with the differential
\[
d(\gamma)=\sum_{i=1}^n(-1)^i(\partial^n_{i,l}(\gamma)-\partial^n_{i,r}(\gamma)),
\]
and consider the complex
\[
    C^n(G,m,l;A)=\{\phi\in Hom(\widetilde{\mathcal C}_n(G,m,l),A)\mid \phi|_{{\mathcal D}_n(G,m,l)}=0\}
\]
with the differential $(d\phi)(\gamma)=\phi(d\gamma)$.

\begin{definition}\label{def:homotopy_crossing_homology}
    The homology $H_*(G,m,l;A)=H(C_*(G,m,l;,A))$ and $H^*(G,m,l;A)=H(C^*(G,m,l;A))$ are called the \emph{$(m,l)$-homology} and \emph{$(m,l)$-cohomology} of the group $G$ with coefficients in $A$.
\end{definition}

Consider the involution $c\mapsto \bar c$ on  $\tilde{\mathcal C}_{n}$ defined by the formula
\[
 \overline{(\sigma,g_0,\dots,g_n)}=(\bar\sigma,g_n^{-1},mg_{n-1}^{-1},\dots,mg_0^{-1})
\]
where $\bar\sigma(k)=n+1-\sigma(k)$, $1\le k\le n$.

Let $\iota(c\otimes a)=(-1)^{\frac{n(n+1)}2}\bar c\otimes a$, $c\otimes a\in C_n(G,m,l;A)$. Denote $C^\pm_*(G,m,l;A)=C_*(G,m,l;A)/(id\mp\iota)$.

\begin{definition}
    The homology $H^\pm_*(G,m,l;A)$ of the complex $C^\pm_*(G,m,l;A)$ is called the \emph{positive} (\emph{negative}) $(m,l)$-homology of $G$ with coefficients in $A$.

Analogously, one defines the positive $(m,l)$-cohomology $H_+^*(G,m,l;A)$ and the negative $(m,l)$-homology $H_+^*(G,m,l;A)$ of the group $G$.
\end{definition}

In the case $A=\Z$, we will omit the coefficient group and write $H_*(G,m,l)$, $H^\pm_*(G,m,l)$ etc.

\begin{example}\label{exa:S3_ml_homology}
Let $G=\Sigma_3$ be the permutation group, $m=(12)$ and $l=1$. The $(m,l)$-homology of the group is the following.
\[
\begin{array}{|c|c|c|c|}
\hline
    n & H_n(G,m,l) & H^+_n(G,m,l) & H^-_n(G,m,l)\\
    \hline
    0 & \Z & \Z & \Z_2 \\
    1 & \Z & 0 & \Z \\
    2 & \Z_3 & \Z_2^2 & \Z_3 \\
    3 & \Z_2^2\oplus\Z_6^2 & \Z_2^2\oplus\Z_6 & \Z_2^2\oplus\Z_6 \\
    4 & \Z^{12}\oplus\Z_3^3 & \Z^{7}\oplus\Z_6^2 &  \Z^{5}\oplus\Z_2^7 \oplus\Z_6 \\
    \hline
\end{array}
\]
\end{example}

Let us show how multicrossing homology relates to $(m,l)$-homology.

Let $K$ be an oriented knot which lies inside a ball $B\subset\R^3$. Let $N(K)$ be a tubular neighborhood of $K$, and $M_K=\overline{B\setminus N(K)}$.

Choose a point $x_0\in S=\partial B$. Denote $G_K=\pi_1(M_K,x_0)$. Then choose a point $y_0\in\partial N(K)$ and a path $\zeta$ in $M_K$ connecting $x_0$ with $y_0$. Let $\mu,\lambda\subset\partial N(K)$ be a meridian and a longitude of $K$ which pass through $y_0$. Consider the homotopy classes $m_K=[\zeta\mu\zeta^{-1}]$ and $l_K=[\zeta\lambda\zeta^{-1}]$ in $\pi_1(M_K,x_0)$.

\begin{proposition}\label{prop:multicrossing_is_ml_homology}
For any abelian group $A$, the homotopy multicrossing homology $H_*(K,A)$ is isomorphic to the $(m,l)$-homology $H_*(G_K,m_K,l_K;A)$.
\end{proposition}
\begin{proof}
    We construct an isomorphism between the multicrossing complex and the $(m,l)$-complex.

Let $\gamma=(\gamma_0,\dots,\gamma_n)\in\widetilde{\mathcal C}(L,n)$. We can suppose that $\gamma_0(0)=\gamma_n(1)=x_0$. Let $\sigma(\gamma)$ be the sequence of numbers of the meridians of $\gamma$ that one encounters while traveling along the knot from the point $y_0$. Contract the left arcs of the meridians $\mu_{i,l}$ to points so that $\gamma_{i-1}(1)=\gamma_i(0)$ and move these point in $\partial N(K)$ to $y_0$. Then use the path $\zeta$ to get loops $\tilde\gamma_0,\dots,\tilde\gamma_n$ based in $x_0$.
Denote
\[
\Phi(\gamma)=(\sigma(\gamma), \tilde\gamma_0,\dots,\tilde\gamma_n)\in\tilde{\mathcal C}_n(G_K,m_K,l_K).
\]
Then $\Phi$ is a bijection between $\widetilde{\mathcal C}(L,n)$ and $\tilde{\mathcal C}_n(G_K,m_K,l_K)$. Moreover, $\Phi(\widetilde{\mathdutchcal D}(L,n))={\mathcal D}_n(G_K,m_K,l_K)$, and $\Phi$ commutes with maps $\partial_{i,l}, \partial_{i,r}$. Hence, $\Phi$ defines an isomorphism of complexes $C_*(L,A)$ and $C_*(G_K,m_K,l_K;A)$ and induces an isomorphism of the homologies.
\end{proof}

\begin{remark}\label{rem:multicrossing_is_ml_homology}
1. The elements $m_K,l_K$ and the isomorphism $\Phi$ depend on the homotopy class of the path $\zeta$. For another choice $\zeta'$ of the path, one gets elements $m'_K,l'_K$ and an isomorphism $\Phi'$  which are conjugated to  $m_K,l_K$ and $\Phi$ by the homotopy class $[\zeta'\zeta^{-1}]\in\pi_1(M_K,x_0)$.

2. The secant classes $\sigma_i(L)$, $i=2,3,4$, can be identified with $(m_K,l_K)$-homology classes in $H_2(G_K,m_K,l_K)$, $H_3^-(G_K,m_K,l_K)$, $H_4^+(G_K,m_K,l_K)$
\end{remark}

\subsection{Secant $(m,l)$-invariants}

Let $K$ be an oriented knot and $H_*(G_K,m_K,l_K)$ its multicrossing homology in the form of $(m,l)$-homology.

Let $G$ be a finite group and $m,l\in G$ commuting elements. Denote
\[
Hom_{m,l}(G_K,G)=\{\phi\in Hom(G_K,G)\mid \phi(m_K)=m,\ \phi(l_K)=l\}.
\]

\begin{definition}\label{def:secant_ml_class}
    The \emph{secant $(m,l)$-invariants} of the knot $K$ are defined by the formula
\begin{equation}\label{eq:secant_ml_invariant}
    \sigma_i(K;m,l)=\sum_{\phi\in Hom_{m,l}(G_K,G)} \phi(\sigma_i(K)),\quad i=2,3,4.
\end{equation}
Here $\sigma_2(K;m,l)$ is considered as a multi-subset of $H_2(G,m,l)$, i.e. an element of $\Z_{\ge 0}[H_2(G,m,l)]$, $\sigma_3(K;m,l)$ is considered as a multi-subset of $H_3^-(G,m,l)$, and $\sigma_4(K;m,l)$ is considered as a multi-subset of $H_4^+(G,m,l)$.
\end{definition}

\begin{theorem}\label{thm:secant_ml_invariant}
    The secant $(m,l)$-invariants $\sigma_i(K;m,l)$, $i=2,3,4$, are knot invariants.
\end{theorem}
\begin{proof}
    Consider the aggregated sum
\[
\sigma_2(K,G)=\sum_{(m,l)}\sigma_2(K;m,l)\in \bigoplus_{(m,l)}\Z_{\ge 0}[H_2(G,m,l)]
\]
and prove that it is invariant. By definition
\[
\sigma_2(K,G)=\sum_{\phi\in Hom(G_K,G)} \phi(\sigma_2(K)).
\]

Let $K'$ be an isotopic knot and $f$ an isotopy from $K'$ to $K$. The isotopy $f$ moves the path $\zeta_{K'}$ that was used to identify the multicrossing homology and the $(m,l)$-homology of $K'$ to a path $\zeta'$ in $M_K$. Then $\Psi=Ad_{[\zeta(\zeta')^{-1}]}\circ f_*$ is an isomorphism between $H_*(G_{K'},m_{K'},l_{K'})$ and $H_*(G_K,m_K,l_K)$ which identifies the secant classes. On the other hands, $\Psi$ induces a bijection between $Hom(G_{K'},G)$ and $Hom(G_K,G)$. Hence,
\begin{multline*}
\sigma_2(K',G)=\sum_{\phi'\in Hom(G_{K'},G)} \phi'(\sigma_2(K'))=
\sum_{\phi\in Hom(G_K,G)} (\phi\circ\Psi)(\sigma_2(K'))=\\
\sum_{\phi\in Hom(G_K,G)} \phi(\sigma_2(K))=\sigma_2(K,G).
\end{multline*}
Thus, $\sigma_2(K,G)$ is invariant. Then $\sigma_2(K;m,l)$ is invariant as the projection of $\sigma_2(K,G)$ into $\Z_{\ge 0}[H_2(G,m,l)]$.

The proof for $\sigma_3(K;m,l)$ and $\sigma_4(K;m,l)$ is analogous.
\end{proof}

\begin{remark}\label{rem:coloring_polynomial}
Fix an element $m\in G$. Consider an augmentation map
\[
\eta_m\colon\bigoplus_{(m',l')}\Z[H_2(G,m',l')]\to\Z[G]
\]
given by the formula
\[
\eta_m(x)=\left\{\begin{array}{cl}
    1\cdot l', & m'=m \\
    0, & m'\ne m,
\end{array}\right.\quad x\in H_2(G,m',l').
\]
Then $P^m_G(K)=\eta_m(\sigma_2(K,G))$ is the coloring polynomial defined by Eisermann~\cite{E}. As a consequence, $(m,l)$-invariants distinguish the Kinoshita-Terasaka knot and the Conway knot, and can detect non-invertibility of knots (see examples in~\cite{E}).
\end{remark}

\begin{example}\label{trefoil_2ml_invariant}
Let $G=\Sigma_3$, $m=(12)$, $l=1$, and $K$ be the right trefoil. Calculations in GAP show that the  $\sigma_2(K;m,l)=2\cdot a$ where $a\in H_2(G,m,l)=\Z_3$ is a non-zero element. Moreover, $\sigma_2(\bar K;m,l)=2\cdot (-a)$, hence, the invariant distinguishes the left and the right trefoils. Note that $P^m_G(K)=P^m_G(\bar K)$, so in this case the $(m,l)$-invariant is stronger than the coloring polynomial.
\end{example}

\begin{example}\label{trefoil_3ml_invariant}
Let $G=A_4$, $m=(123)$, $l=1$, and $K$ be the right trefoil. Then $\sigma_3(K;m,l)=3\cdot a$ where $a\in H^-_3(G,m,l)=\Z_2^5\oplus\Z_4$ is an element of order $2$.
\end{example}

\begin{example}\label{trefoil_4ml_invariant}
Let $G=\Sigma_3$, $m=(12)$, $l=1$, and $K$ be the right trefoil. The standard trefoil diagram has three positive $4$-secants (Fig.~\ref{fig:trefoil_4secants}).
\begin{figure}[h]
    \centering
    \includegraphics[width=0.3\linewidth]{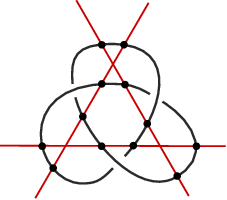}
    \caption{The right trefoil and its $4$-secants}
    \label{fig:trefoil_4secants}
\end{figure}

Then $\sigma_4(K;m,l)=2\cdot a$ where $a\in H^+_4(G,m,l)=\Z^7\oplus\Z_6^2$ is an element of order $3$.
\end{example}

\section*{Acknowledgements}

The author is grateful to Vassily Olegovich Manturov for fruitful discussions.


\end{document}